\documentclass[review,hidelinks,onefignum,onetabnum]{siamart251104}

\usepackage{lipsum}
\usepackage{amsfonts}
\usepackage{epstopdf}
\usepackage{algorithmic}
\ifpdf
  \DeclareGraphicsExtensions{.eps,.pdf,.png,.jpg}
\else
  \DeclareGraphicsExtensions{.eps}
\fi

\usepackage{amsmath,amssymb,amsfonts}
\usepackage{graphicx,color}
\usepackage[utf8]{inputenc}
\usepackage{enumerate}
\usepackage{mathtools}
\usepackage{cases}
\usepackage{multicol}
\usepackage{tabularx}
\usepackage{adjustbox}
\usepackage{varwidth} 
\usepackage{bm}
\usepackage{tikz,pgfplots}
\usetikzlibrary{calc}
\usetikzlibrary{decorations.markings}
\tikzset{>=latex} 
\usepackage{array}
\usepackage{hyperref}
\usepackage{cleveref} 

\usepackage{pgfplots}
\pgfplotsset{compat=1.18} 

\newtheorem{thm}{Theorem}[section]
\newtheorem{lem}[thm]{Lemma}
\newtheorem{defn}[thm]{Definition}
\newtheorem{prop}[thm]{Proposition}
\newtheorem{cor}[thm]{Corollary}

\newtheorem{assumption}[thm]{Assumption}

\newtheorem{remark}[thm]{Remark}

\newcommand{\A}{\mathcal{A}}

\newcommand{\BB}{\mathcal{B}}
\newcommand{\CC}{\mathcal{C}}
\newcommand{\G}{\mathcal{G}}
\renewcommand{\Re}{\mathrm{Re}}

\headers{Well-posedness of 1D Hyperbolic control systems}{Bouchra Elghazi, Birgit Jacob, and Hans Zwart}

\title{Well-posedness and controllability of hyperbolic boundary control systems on a one-dimensional spatial domain \thanks{Submitted to the editors on December 30 , 2025.
\funding{This work was supported and funded by the European Union (Horizon Europe MSCA project ModConFlex, grant number 101073558) and by the Deutsche Forschungsgemeinschaft (DFG, German Research Foundation) Project-ID 531152215 -- CRC 1701.}}}

\author{Bouchra Elghazi
\thanks{University of Wuppertal, School of Mathematics and Natural Sciences Gaußstraße 20, 42119 Wuppertal, Germany, (\email{elghazi@uni-wuppertal.de}, \email{bjacob@uni-wuppertal.de}).}
\and Birgit Jacob\footnotemark[2]
\and Hans Zwart
\thanks{Department of Applied Mathematics, University of Twente, P.O. Box 217, 7500 AE, Enschede, The Netherlands and Department of Mechanical Engineering, Eindhoven University of Technology, P.O. Box 513, 5600 MB, Eindhoven, The Netherlands, \email{h.j.zwart@utwente.nl}.}
}

\ifpdf
\hypersetup{ pdftitle={well-posedness and controllability of hyperbolic boundary control systems on a 1D spatial domain} }
\fi

\begin{document}

\maketitle


\begin{abstract}
    The aim of this paper is to investigate the well-posedness of a class of boundary control and observation systems on a one dimensional spatial domain. We derive a necessary and sufficient condition characterizing the well-posedness of these systems. Furthermore, we show that the well-posedness and full control and observation implies exact controllability and exact observability. The theoretical results are illustrated using Euler-Bernoulli beam models.
\end{abstract}



\section{Introduction}
Many physical structures rarely remain rigid when subjected to force. In fact, the modelling and analysis of flexible structures like beams, or strings play a crucial role in many engineering applications such as in robotics, aerospace, and high precision machines. These structures are typically modelled by partial differential equations (PDEs), particularly, linear time-invariant PDEs have garnered significant attention. In many applications, these PDEs are controlled and observed only at their boundaries. They are naturally modelled within the framework of boundary control and observation systems. This concept was initiated by Fattorini in the 1960s \cite{Fattorini:68}, establishing a starting point for studying linear time-invariant systems governed by PDEs with boundary control. Later researchers such as Salamon \cite{Salamon:87,salamon_realization:88}, Weiss \cite{weiss:89,Weis:89,Wei:94,weiss1989representation,weiss1994transfer} and Curtain \cite{curtain1997salamon} developed Fattorini’s early work to the modern theory of well-posedness, admissibility and regular linear systems. 

The concept of well-posedness is fundamental in the study of boundary control and observation, as it forms the basis for further control and stability analysis. Over the years, the study of well-posedness of linear time-invariant systems has attracted considerable attention, leading to a significant amount of literature. Seminal contributions in this area include the works by Lasiecka and Triggiani \cite{LasieckaTriggiani2000a,LasieckaTriggiani2000b} and by Staffans \cite{Sta:05}. For further reading, we also refer to \cite{MorrisCheng:99,MorrisCheng:03,CurtainWeiss1989,CurtZwa:20,JacZw:12,TucsnakWeiss:14,Weiss2001}.
Informally speaking, well-posedness refers to the property that for every initial condition in the state space and any input function in a specified space of functions, the system has a unique state trajectory and a unique output function. Moreover, the output must belong to a specified space of functions, and both the state and the output must depend continuously on the initial state and on the input. Equivalently, it concerns the well-definedness and boundedness of the four mappings from initial state to final state, initial state to output, input to final state, and input to output (represented by \emph{the transfer function}). For boundary control and observation systems that are impedance passive, there exists a simple way of characterizing the well-posedness. In fact, if the system is internally well-posed, then the well-posedness is equivalent to the boundedness of the transfer function on a vertical line in the open right half plane. Notable examples within this class are port-Hamiltonian systems which provide a powerful framework for modelling as it views physical systems from an energy-based perspective \cite{Schaft:06,SchJelt:14,SchMas:02}. 

The primary focus of this article is to investigate the well-posedness of the following class of linear time-invariant systems on a one-dimensional spatial domain
\begin{equation}\label{eqn:system1}
\begin{aligned}
     \frac{\partial x}{\partial t}(\zeta,t) &= \left( P_2 \frac{\partial^2}{\partial \zeta^2} +  P_1 \frac{\partial}{\partial \zeta} + P_0(\zeta) \right)\mathcal{H}(\zeta) x(\zeta,t), \quad t>0, \quad \zeta \in (0,1) \\
     u(t) &= W_{B,1} \tau(\mathcal H x)(t) , \quad  0 = W_{B,2} \tau(\mathcal H x)(t), \quad t>0 \\
    y(t) &= W_C \tau (\mathcal H x)(t), \quad t>0 \\
    x(\zeta,0) &= x_0(\zeta), \quad \zeta \in [0,1],
\end{aligned}
\end{equation}
where $x(\zeta,t) \in \mathbb F^n$ ($\mathbb F=\mathbb R$ or $\mathbb C$), $u(t) , \, y(t) \in \mathbb F^m$ and $\tau$ is the trace operator, given by
\begin{equation}\label{eqn:trace}
    \tau(x) =
    \begin{pmatrix}
        x(1)  & x'(1) & x(0)  & x^{\prime}(0)
    \end{pmatrix}^\top .
\end{equation}
We assume that $P_0$, $P_1$, $P_2$ and ${\mathcal H}(\zeta)$ are $n\times n$-matrices, where $P_1$ is self-adjoint,  $P_2$ is skew-adjoint and invertible and ${\mathcal H}(\zeta)$ is selfadjoint and positive. Further,  $W_{B,1}$, $W_{B,2}$, $W_C$ are matrices of suitable sizes.
The state space is given by $X = \mathrm{L^2}((0,1);\mathbb F^n)$ with energy inner product
\begin{equation*}
    \left\langle f, g \right\rangle_{X} = \frac{1}{2} \int_0^1 g(\zeta)^\ast \mathcal H(\zeta) f(\zeta) d\zeta. 
\end{equation*} 
For simplicity, we restrict attention to the interval $[0,1]$, similar results hold for an arbitrary compact interval. This class of systems covers in particular the Schr\"odinger equation and the Euler-Bernoulli beam equations.

Using the port-Hamiltonian approach, internal well-posedness \cite{GorZwaMas:2005,Villegas:07}, stability and stabilizability \cite{Aug:16,Aug:15,AugJac:14,schmid2021stabilization}, observer design \cite{LeGorrecToledoRamirezWu:23} and robust output regulation \cite{HumLassi:18,humaloja2016robust,PaunLeGorrecRamirez:21} of system \eqref{eqn:system1} have been investigated. Here, internal well-posedness refers to the well-definedness and boundedness of the mapping from the initial state to the final state. The well-posedness of system \eqref{eqn:system1} has remained an open problem for a long time and will be addressed in this paper. Our main result (Theorem \ref{0main_result}) provides an equivalent characterization in terms of a matrix condition in the case $m=2n$. In the general setting, Corollary \ref{corollary1} establishes a sufficient condition for well-posedness, again formulated in terms of a matrix condition. Partial results were obtained in \cite{ElghJacZw:25}, where the assumptions $P_1=0$ and constant $\mathcal H$ were imposed. The proof in \cite{ElghJacZw:25} relies on a diagonalization technique, which is not available in the general setting. Moreover, it was shown in \cite{ElghJacZw:25} that, for system \eqref{eqn:system1}, internal well-posedness does not, in general, imply well-posedness. Further, well-posedness under the assumption $P_2=0$ and $P_1$ is invertible has been studied \cite{JacZw:12,ZwGorMasVill:10}. In this situation internal well-posedness is equivalent to well-posedness. Beside well-posedness we characterise exact controllability and exact observability of system \eqref{eqn:system1}. In Theorem \ref{mainresult2} we prove that under the assumption $m=2n$ and $P_0(\zeta)^\ast=-P_0(\zeta)$ the well-posed of system \eqref{eqn:system1} implies exact controllability and exact observability. A similar result was obtained in \cite{JacKai} under the assumptions $P_2=0$ and $P_1$ is invertible.

\medskip
We proceed as follows. In the next section, we provide the necessary mathematical background on boundary control and observation systems and on well-posedness. In Section~\ref{section3}, the main result on well-posedness is presented, while Section \ref{section4} is devoted to exact controllability and observability. Applications for Euler-Bernoulli beam models are presented in Section~\ref{section5}. Concluding remarks and possible future topics are given in Section~\ref{section6}.

\medskip
\textbf{Notation.} Let $\mathbb F\in \{ \mathbb R ,\mathbb C\}$, with $\|\cdot\|$ denoting the Euclidean norm of $\mathbb F^{n}$. The state of the system at time $t$ and spatial position $\zeta$ is denoted by $x(\zeta,t)$, the derivative w.r.t.\ time $t$ is denoted by $\dot{x}$, while $x'$ refers to the derivative of $x$ w.r.t.\ $\zeta$. Moreover, $u(t) \in U$ and $y(t) \in Y$ denote the inputs and the outputs, respectively.  
For simplicity, we write $\mathcal H(\zeta) x(\zeta,t)$ as $(\mathcal H x)(\zeta,t)$.
Throughout, we denote by $\mathbb C^{+}_{\alpha}$ the right half-plane $\left\{s \in \mathbb C~\vert\,  \Re(s) > \alpha \right\}$ and the symbol $\lesssim$ is used to indicate an inequality up to a constant multiple $c>0$, where $c$ is generic and may vary from line to line.
For a linear (unbounded) operator $A$ on a Hilbert space $X$ with domain $D(A)$, we denote by $\rho(A)$ its resolvent set and by $\ker A$ its kernel. If $A$ is a densely defined linear operator, then its adjoint is denoted by $A^\ast: D(A^\ast) \subseteq X \to X$ and its (extrapolated) extension is denoted by $A_{-1}$, see \cite{engel2000one,pazy2012semigroups} for a formal definition. For Lebesgue and Sobolev spaces, we adopt the standard notation, such as in \cite{adams}. The space of bounded, linear operators between two Hilbert spaces $X$ to $Z$ is denoted by $\mathcal{L}(X,Z)$, while $C^k(X;Z)$ is the space of all functions on $X$ mapping into $Z$ that are $k$-times continuously differentiable. Henceforth, we abbreviate  $\mathcal L(X):= \mathcal L(X,X)$.
Additionally, for a positive bounded operator $A$, $A^{\frac{1}{2}}$ will represent the square root of $A$, see \cite{rudin1991functional} for a formal definition.

\section{Boundary control and observation systems}
In this section, we provide some background and useful preliminary results on boundary control and boundary observation systems of the form
\begin{equation}\label{BCS}
    \begin{aligned}
        \dot{x}(t) &= \A x(t), \quad x(0)=x_0, \\
        u(t) &= \BB x(t), \\
        y(t) &= \CC x(t).
    \end{aligned}
\end{equation}
We briefly review some important definitions and facts based on 
\cite{CurtZwa:20,JacZw:12,Sta:05,tucsnak2009observation}. 

\begin{defn}\cite[Definition 10.1.2]{CurtZwa:20}\label{def_BCS}
Let $X$, $U$ and $Y$ be complex Hilbert spaces. We call \eqref{BCS}, denoted by $(\A,\BB,\CC)$,
a boundary control and observation system on $(X,U,Y)$ if the following properties hold
\begin{enumerate}
    \item $\A : D(\A)\subset X \to X$, $\BB : D(\BB) \subset X \to U$, $\CC : D(\A) \subset X \to Y$ are linear operators with $D(\A) \subseteq D(\BB)$.
    \item The operator $A_0 :D(A_0) \to X$ defined by
    $$ A_0 x = \A x \quad \text{for } x \in D(A_0) := D(\A) \cap \ker(\BB) $$
    is the generator of a $C_0$-semigroup $(T(t))_{t\geq0}$ on $X$.
   \item There exists an operator $B \in \mathcal L(U, X)$ such that for all $u \in U$ we have $Bu \in D(\A)$, ${\A B \in \mathcal L(U, X)}$ and $$ \BB Bu = u, \quad u \in U. $$
  \item The operator $\CC$ is bounded from $D(A_0)$ to $Y$. Here, $D(A_0)$ is equipped with the graph norm.
\end{enumerate}
\end{defn}

Next, we define classical solutions for boundary control and observation systems.
\begin{defn}\cite[Definition 10.1.3]{CurtZwa:20}
    The pair $(x,y)$ is a classical solution of the boundary control and observation system \eqref{BCS} on $[0,T]$ if for $x_{0} \in D(\mathcal A)$ and $u \in C^{2}([0,T]; U)$, the pair $x \in C^1([0,T];X)$ and $y \in C([0,T];Y)$ satisfies \eqref{BCS} for all $t \in [0,T]$.
\end{defn}

For sufficiently smooth inputs, it is possible to reformulate the boundary control and observation system  into an abstract differential equation of the form
\begin{equation}\label{BCS_linear}
\begin{aligned}
    \dot{v}(t) &= A_0 v(t) - B\dot{u}(t) + \A B u(t),  \\
    v(0) &= v_0, 
\end{aligned}
\end{equation}
with the output defined by
\begin{equation}\label{output_y_v}
    y_{v}(t) = \mathcal{C} v(t) .
\end{equation}   
Since $A_0$ is the infinitesimal generator of a $C_0$-semigroup and $B$ and $\A B$ are bounded, the Cauchy problem \eqref{BCS_linear} has a unique classical solution for $v_0 \in D(A_0)$ and $u \in C^2([0,T]; U)$. Therefore, we have the following relation between the classical solutions of \eqref{BCS} and \eqref{BCS_linear}.

\begin{thm}\label{thm:classical_solution}\cite[Theorem 10.1.4]{CurtZwa:20}
    Consider the boundary control and observation system \eqref{BCS} and the abstract differential equation \eqref{BCS_linear} with the output \eqref{output_y_v}. Assume that $u \in C^2([0,T]; U)$. Then, if $v_0 = x_0 - Bu(0) \in D(A_0)$, the classical solutions of \eqref{BCS} and \eqref{BCS_linear}--\,\eqref{output_y_v} are related by 
    \begin{align*}
        v(t) &= x(t) - B u(t), \\
        y_{v}(t) &= y(t) - \mathcal{C} B u(t) .
    \end{align*}
   Furthermore, the classical solution of \eqref{BCS} is unique.
\end{thm}

Next, we provide a formal definition of well-posedness.
\begin{defn}\cite[Definition 13.1.3]{JacZw:12}
    The boundary control and observation system \eqref{BCS} is called well-posed, if there exist $t, m >0$ such that every classical solution of \eqref{BCS} satisfies
     \begin{equation}\label{well-posed}
         \|x(t)\|_{X}^2 + \int_0^{t} \|y(s)\|^2 ds \leq m \left( \|x_0\|_{X}^2 + \int_0^{t} \|u(s)\|^2 ds \right).
     \end{equation}
\end{defn}

We note that if the inequality \eqref{well-posed} holds for one $t>0$, then it holds for all $t>0$ \cite{JacZw:12}. Moreover, well-posedness is equivalent to the existence of \emph{mild solutions} for an arbitrary initial condition $x_0 \in X$ and an arbitrary input $u \in \mathrm L^2((0,t);U)$, such that $x$ is continuous and $y \in \mathrm L^2((0,t);Y)$ \cite[Section 13.1]{JacZw:12}. Further, the state and output depend continuously on the initial state and input function. 

\begin{remark}\cite[Section 11.1 and Theorem 11.2.1]{JacZw:12} Assume that the boundary control and observation system \eqref{BCS} is well-posed.
If $x_0 \in X$ and  $u \in \mathrm H^1([0,T];U)$, then the mild solution of system \eqref{BCS} is given by 
    \begin{align*}
        x(t) &= T(t)(x_0 - Bu(0)) + \int_0^t T(t-s) \left( \A Bu(s) - B \dot{u}(s) \right) ds + Bu(t).
    \end{align*}
In addition, if $u \in C^2([0,T]; U)$ and $x_0 - Bu(0) \in D(A_0)$, then the classical solution is a mild solution and the corresponding output is given by  
    \begin{align*}
         y(t) &= \mathcal{C} T(t)(x_0 - Bu(0)) + \mathcal{C} \int_0^t T(t-s) \left( \A Bu(s) - B \dot{u}(s) \right) ds + \mathcal{C} Bu(t).
    \end{align*}
\end{remark}

\begin{defn}\cite[Definition 3.2.12]{Aug:16}
    The boundary control and observation system \eqref{BCS} on $(X,U,U)$ is called impedance passive if, 
    \begin{equation}\label{eq:impedance passive}
        \Re\left\langle\A x, x \right\rangle_X \leq \Re\left\langle \BB x, \CC x \right\rangle_{U}, \quad x \in D(\A) .
    \end{equation}
    It is called impedance energy-preserving if 
    \begin{equation}
        \Re \left\langle\A x, x \right\rangle_X = \Re \left\langle \BB x, \CC x \right\rangle_{U}, \quad x \in D(\A) .
    \end{equation}
\end{defn}

\begin{defn}\cite[Definition 12.1.1]{JacZw:12}\label{defn:transfer_fct}
    Let $s \in \mathbb C$ and $u_0 \in U$. The triple $\left( u(t) ,x(t),y(t)\right)_{t\geq 0}$ is called an exponential solution of the boundary control and observation system \eqref{BCS} if there exist $x_0 \in X$, $y_0 \in Y$, such that 
    \begin{equation}\label{eqn:George2}
        \left( u(t), x(t), y(t) \right) = \left( u_0 e^{st}, x_0 e^{st}, y_0 e^{st}\right), \quad \text{for a.e. } t \geq 0,
    \end{equation}
    and $(x,y)$ is the mild solution of the boundary control and observation system \eqref{BCS} with the input function $u(t)= u_0 e^{st} \in C^1([0,t]; U)$ and initial condition $x_0$.
    
Let $s \in \mathbb C$. If for every $u_0 \in U$ there exists an exponential solution, and
the corresponding output trajectory $y_0 e^{s t}, \, t \in [0,\infty)$ is unique, then we call the mapping $u_0 \mapsto y_0$ the transfer function at $s$, denoted by $\G(s)$.
The mapping $s \in \rho(A_0) \mapsto \G(s)$ exists for every $s \in \rho(A_0)$ \cite{JacZw:12}. 
\end{defn} 

\begin{thm}\label{thm:thm2.8}\cite[Theorem 12.1.3]{JacZw:12}
The transfer function of the boundary control and observation system \eqref{BCS} is given by
\begin{equation*}
    \G(s) = \CC(sI - A_0)^{-1} (\A B - sB) + \CC B, \quad s \in \rho(A_0).
\end{equation*}
Moreover, every exponential solution of \eqref{BCS} is also a classical solution. Furthermore, for $s \in \rho(A_0)$ and $u_0 \in U$, $\G(s)u_0$ can also be calculated as the (unique) solution of
\begin{equation}\label{syst:transfer_fct}
    s x_0 = \A x_0, \qquad
    u_0 = \BB x_0, \qquad
    \G(s)u_0 = \CC x_0,
\end{equation}
with $x_0 \in D(\A)$. Further, $x_0\in D(\A)$ is uniquely determined by \eqref{syst:transfer_fct}.
\end{thm}

\begin{corollary}\label{cor:passive}
Let $u_0 \in U$ and $x_0 \in D(\mathcal A)$ be the unique solution of \eqref{syst:transfer_fct}.
If the boundary control and observation system is impedance passive, then we have the following inequality
    \begin{equation}
        \Re(s) \|x_0\|^2_X \leq \Re \left\langle u_0, \G(s) u_0 \right\rangle, 
    \end{equation}
for all $s$ such that $\G(s)$ exists.
\end{corollary}
\begin{proof}
By Definition \ref{defn:transfer_fct}, the transfer function is related to the exponential solution \eqref{eqn:George2} via
$$ \left(u(t), x(t), y(t) \right) = \left( u_0 e^{st}, x_0 e^{st}, \G(s) u_0 e^{st} \right), \quad \text{for a.e. } t\geq 0. $$
By Theorem \ref{thm:thm2.8} $x_{0} \in D(\mathcal A)$ and \eqref{syst:transfer_fct} is satisfied. The impedance passivity of the system gives 
\begin{align*}
    \Re \left\langle \mathcal{A} x_0 e^{st}, x_0 e^{st} \right\rangle_X 
    &\leq \Re \left\langle \mathcal B x_{0} e^{st}, \mathcal C x_{0} e^{st} \right\rangle \\
    &= \left\langle u_0 e^{st}, \G(s) u_0 e^{st} \right\rangle.
\end{align*}  
Using \eqref{syst:transfer_fct}, we get
$$ \Re \left\langle s x_0 e^{st}, x_0 e^{st} \right\rangle_X \leq \Re \left\langle u_0 e^{st}, \G(s) u_0 e^{st} \right\rangle. $$
 We divide by $e^{2\Re(s)t}$ and obtain
\begin{align*}
   \Re(s) \|x_0\|_X^2 \leq \Re \left\langle u_0, \G(s) u_0 \right\rangle. 
\end{align*} 

\end{proof}

Whereas well-posedness requires the inequality \eqref{well-posed} to be satisfied, for impedance passive systems, well-posedness is already determined by the boundedness of the transfer function on some vertical line in the open right half-plane. We have the following result.
\begin{thm}\label{thm:well-posedness}\cite[Theorem 5.1]{Sta:02}
    An impedance passive boundary control and observation system is well-posed if and only if its transfer function $\G$ is bounded on some vertical line in the open right half plane $ \mathbb C^+_{0}$, i.e.~$\G :  \mathbb C^+_{0} \longrightarrow \mathcal L(U)$ satisfies
    $$ \exists~ r >0 \text{ such that } \sup_{\omega \in \mathbb R} \| \G(r + i\omega)\| <\infty. $$
\end{thm}

For well-posed boundary control and observation systems, even if the transfer function is bounded on some right half plane, this does not imply that the limit as $\Re(s) \to \infty$ exists. Well-posed systems with this property are called regular.

\begin{defn}\cite[Definition 4.4]{Wei:94}
Let $\G$ be the transfer function of a well-posed boundary control and observation system. The boundary control and observation system is called regular if $\displaystyle\lim_{s \in \mathbb{R}, s \to \infty} \G(s)$ exists. 
If the boundary control system is regular, then the feedthrough term $D$ is defined as $D = \underset{s \in \mathbb R, s \to \infty}{\lim} \G(s)$.
\end{defn}

\medskip
Next, we state that well-posed boundary control and observation systems are stable under bounded perturbations and admissible feedback operators.
\begin{lem}\label{lem:perturbation}
    Let $\G$ be the transfer function of the boundary control and observation system $\left( \A, \BB, \CC\right)$, and let $P$ be a bounded linear operator on $X$. Then the boundary control and observation system $\left( \A, \BB, \CC\right)$ is well-posed if and only if the the boundary control and observation system $\left(\mathcal{A}+P, \BB,\CC\right)$ is well-posed. Moreover, if $\G_P$ is the transfer function of the boundary control and observation system $\left( \A +P, \BB, \CC\right)$, then $\displaystyle\lim_{s \in \mathbb R, s \to \infty} \G(s)$ exists if and only if $\displaystyle\lim_{s \in \mathbb R, s \to \infty} \G_P(s)$ exists and 
    \begin{equation*}
        \lim_{s \in \mathbb R, s \to \infty} \G(s) = \lim_{s \in \mathbb R, s \to \infty} \G_P(s) .
    \end{equation*}
    Moreover, if $\displaystyle\lim_{\Re(s)\to\infty} \G(s)$ exists, then 
    \begin{equation*}
        \lim_{\Re(s) \to \infty} \G(s) = \lim_{\Re(s) \to \infty} \G_P(s) .
    \end{equation*}
\end{lem}
\begin{proof}
    Let $\left( \A, \BB, \CC\right)$ be a well-posed system. Consider the system 
    \begin{equation}\label{syst:perturbed_BCS}
    \begin{aligned}
        \dot{x}(t) &= (\A + P) x(t), \\
        u(t) &= \BB x(t), \\
        y(t) &= \CC x(t). 
    \end{aligned}
    \end{equation}
    To show that \eqref{syst:perturbed_BCS} is well-posed, we need to show that $A_{P}:= (\A +P)_{|D(\A) \cap \ker \BB}$ generates a $C_{0}$-semigroup, the input and output operators of \eqref{syst:perturbed_BCS} are admissible and the associated transfer function $\G_{P}$ is bounded on a right-half plane. Since $(\A, \BB, \CC)$ is well-posed and $P$ is bounded, $A_{P}:= (\A + P)_{|D(\A) \cap \ker \BB}$ generates a $C_{0}$-semigroup and the input and output operators of \eqref{syst:perturbed_BCS} are admissible see \cite[Theorem 1.3]{engel2000one} and \cite[Remark 2.11.3]{tucsnak2009observation}. In fact, since $A_{0}$ generates a $C_{0}$-semigroup and $P$ is bounded, $A_{P}$ generates a $C_{0}$-semigroup as well \cite[Theorem 1.3]{engel2000one}. Moreover, the class of admissible operators remains unchanged under bounded perturbations \cite[Theorem 5.4.2]{tucsnak2009observation}. Hence, it only remains to show that the transfer function of the perturbed system \eqref{syst:perturbed_BCS} is bounded in some right-half plane and that if \eqref{BCS} is regular then \eqref{syst:perturbed_BCS} is again regular and their transfer functions converge to the same limit as $\Re(s) \to \infty$. To this end, let $s \in \rho(A_{P}) \cap \rho(A_{0})$ and $u_{0} \in U$. The transfer function $\G_{P}$ of the perturbed system \eqref{syst:perturbed_BCS} can be calculated via the set of equations
    \begin{equation*}
        s x_{0} = (\A + P) x_{0}, \qquad
        u_{0} = \BB x_{0}, \qquad
        \G_{P}(s) u_{0} = \CC x_{0}.
    \end{equation*}
    By splitting $x_{0}= p_{0} + q_{0}$ with $p_{0}$ satisfying
    \begin{equation*}
        s p_{0} = \A p_{0}, \qquad
        u_{0} = \BB p_{0}, \qquad
        \G(s) u_{0} = \CC p_{0},
    \end{equation*}
    we see that $q_{0}$ satisfies
    \begin{align*}
        s q_{0} &= (\A + P) q_{0} + P p_{0}, \\
        \BB q_{0} &= 0 .
    \end{align*}
    This implies that $q_{0} \in D(A_{0})=D(A_{P})$ and 
    \begin{equation*}
        q_{0} = (sI - A_{P})^{-1} P  p_{0} .
    \end{equation*}
    Thus, the transfer function $\G_{P}(s)$ is given by
    \begin{align*}
        \G_{P}(s) u_{0} 
        &= \CC p_{0} + \CC q_{0} \\
        &= \G(s) u_{0} + \CC (sI - A_{P})^{-1} P p_{0}.
    \end{align*} 
    From \cite[Section 8]{JacZwa_gamm} and \cite[Theorem 12.1.3]{JacZw:12}, $p_{0}$ is uniquely determined as 
    \begin{equation}\label{eqn:p_0}
        p_{0} = (sI - A_{0})^{-1} \left( \A B - A_{0,-1} B \right) u_{0}.
    \end{equation}
    Let $\tilde B = \left( \A B - A_{0,-1} B \right)$. For $s \in \rho(A_{0}) \cap \rho(A_{P})$, we have
    \begin{equation*}
         \G_{P}(s) u_{0} = \G(s) u_{0} + \CC (sI - A_{P})^{-1} P (sI - A_{0})^{-1} \tilde{B} u_{0} .
    \end{equation*}
    Furthermore, from the admissibility of $\CC$ and the boundedness of $P$, it follows that for some $\alpha>0$, there exists $k>0$ such that for $s \in \mathbb C^{+}_{\alpha}$
    \begin{align*}
        \|\CC (sI - A_{P})^{-1} P p_{0}\| 
        &\le \|\CC (sI - A_{P})^{-1}\| \| P\| \|p_{0}\| \\
        &\le \frac{k}{\sqrt{\Re(s) - \alpha}} \|P\| \|u_{0}\| .
    \end{align*}
    Since $(\A, \BB, \CC)$ is well-posed, the transfer function $\G(s)$ is bounded on $\mathbb C^{+}_{\alpha}$. Therefore, we conclude that $\G_{P}(s)$ is also bounded on $\mathbb C^{+}_{\alpha}$ and thus the perturbed system \eqref{syst:perturbed_BCS} is well-posed. Moreover, if $\displaystyle\lim_{s \in \mathbb R, s\to\infty} \G(s)$ and $\displaystyle\lim_{\Re(s)\to\infty} \G(s)$ exist, then by \eqref{eqn:p_0}, we see that $p_{0}$ converges to zero as $s \in \mathbb R$, $s \to \infty$ and as $\Re(s) \to \infty$. Therefore, we obtain that
    \begin{equation*}
        \lim_{s\in \mathbb R, s\to\infty} \G_{P}(s) u_{0} = \lim_{s \in \mathbb R, s\to\infty} \G(s) u_{0} = \lim_{s \in \mathbb R, s\to\infty} \CC p_{0} ,
    \end{equation*}
    and 
    \begin{equation*}
        \lim_{\Re(s)\to\infty} \G_{P}(s) u_{0} = \lim_{\Re(s)\to\infty} \G(s) u_{0} = \lim_{\Re(s) \to\infty} \CC p_{0}  .
    \end{equation*}
    Thus, we have shown that if $(\A,\BB,\CC)$ is regular, then the perturbed system \eqref{syst:perturbed_BCS} is again regular and the final equality in the lemma also holds
\end{proof}

\begin{thm}\label{thm:feedback}(\cite[Proposition 4.9]{Wei:94} and \cite[Theorem 13.1.12]{JacZw:12})
   Let $\G$ be the transfer function of the well-posed boundary control and observation system \eqref{BCS}.
    Let $F$ be a bounded linear operator from $Y$ to $U$ such that the inverse of $I+\G(s)F$ $\left(\text{or } I+F\G(s)\right)$ exists and is bounded in $s$ in some right half-plane. Then the closed-loop system is again well-posed, that is, the boundary control and observation system
    \begin{align*}
         \dot{x}(t) &= \A x(t) , \\
         u(t) &=(\mathcal{B} + F\mathcal{C})x(t) , \\
         y(t) &= \mathcal{C}x(t),
    \end{align*}
    is again well-posed.
    Such an operator $F$ is called an admissible feedback operator. 
\end{thm}

We conclude the section by a useful lemma.
\begin{lem}\label{lem:invertibility}
    Let $\left(\mathcal A, \mathcal B_0, \mathcal C_0\right)$ and $\left(\mathcal A, \mathcal B_1, \mathcal C_1 \right)$ be two well-posed boundary control and observation systems on $\left(X,U,Y\right)$ and let $Q \coloneqq \begin{pmatrix}
        Q_{11} & Q_{12} \\
        Q_{21} & Q_{22}
    \end{pmatrix} \in \mathcal L(U \times Y)$ be such that 
    \begin{equation}\label{eqn:Q}
        Q \begin{pmatrix}
            \mathcal B_0 x \\
            \mathcal C_0 x
        \end{pmatrix} = \begin{pmatrix}
            \mathcal B_1 x \\
            \mathcal C_1 x
        \end{pmatrix}, \quad x \in D(\mathcal{A}).
    \end{equation}
    We denote by $\G_0$ and $\G_1$ the transfer functions of $\left(\mathcal A, \mathcal B_0, \mathcal C_0\right)$ and $\left(\mathcal A, \mathcal B_1, \mathcal C_1 \right)$, respectively and we assume that $\displaystyle \lim_{\Re(s) \to \infty} \G_0(s) =0$. If $Q$ is invertible, then $Q_{11}$ is invertible.
\end{lem}
\begin{proof} 
Since the two systems $\left(\mathcal A, \mathcal B_0, \mathcal C_0\right)$ and $\left(\mathcal A, \mathcal B_1, \mathcal C_1 \right)$ are well-posed, there exist $\alpha > 0$ and $M_0,  M_1>0$ such that (see \cite{Weiss2001})
\begin{equation*}
    \sup_{\Re(s)>\alpha} \| \G_0(s) \| < M_0 , \quad \text{and }
    \sup_{\Re(s)>\alpha} \| \G_1(s) \| < M_1 .
\end{equation*} 
By Theorem \ref{thm:thm2.8}, for $u_{0} \in U$ and $s \in \mathbb C^+_{\alpha}$, $\G_0(s)u_0$ is the unique solution of \eqref{syst:transfer_fct}.
Therefore, applying the matrix 
$\begin{psmallmatrix}
    Q_{11} & Q_{12} \\
    Q_{21} & Q_{22}
\end{psmallmatrix}$ 
to the input and the output equations we obtain that 
\begin{equation*}
\begin{aligned}
    s x_0  &= \mathcal{A} x_0 , \\
    \begin{pmatrix}
        \mathcal{B}_1 x_0 \\
        \mathcal{C}_1 x_0 
    \end{pmatrix} &= \begin{pmatrix}
        Q_{11} & Q_{12} \\
        Q_{21} & Q_{22}
    \end{pmatrix} \begin{pmatrix}
        u_0  \\
        \G_0(s) u_0 
    \end{pmatrix}.
\end{aligned}
\end{equation*}
Defining 
\begin{align*}
    u_1(s) &= Q_{11} u_0 + Q_{12} \G_0(s) u_0, \\
    y_1(s) &= Q_{21} u_0 + Q_{22} \G_0(s) u_0 ,
\end{align*}
it follows that 
\begin{equation*}
    s x_0  = \mathcal{A} x_0 , \qquad
    \mathcal{B}_1 x_0  = u_1(s) , \qquad
    \mathcal{C}_1 x_0  = y_1(s) .
\end{equation*}
As the transfer function is unique, we have $y_1(s) = \G_1(s) u_1(s)$. Hence, we have shown that for every $u_{0} \in U$ and $s \in \mathbb C^{+}_{\alpha}$ there exists $u_1(s) \in U $ such that
\begin{equation}\label{eqn:2.10}
    \begin{pmatrix}
        u_1(s) \\
        \G_1(s) u_1(s)
    \end{pmatrix} = Q \begin{pmatrix}  
         u_0 \\
        \G_0(s) u_0
    \end{pmatrix} .
\end{equation}
In addition, as $Q$ is invertible, the same argument shows that for every $u_{1} \in U$ and $s \in \mathbb C^{+}_{\alpha}$ there exists $ u_0(s) \in U $ such that 
\begin{equation}\label{eqn:2.11}
    Q \begin{pmatrix}
       u_{0}(s) \\
       \G_0(s) u_{0}(s)
    \end{pmatrix}  = \begin{pmatrix}
       u_{1}\\
       \G_1(s) u_{1}
    \end{pmatrix}  .
\end{equation}
Using the equation \eqref{eqn:2.10}, for $s \in \mathbb C^{+}_{\alpha}$ and $u_{0} \in U$ we have
\begin{align*}
 M_1 \| Q_{11} u_{0} \| + M_1 \| Q_{12} \G_0(s) u_{0} \| 
        &\ge M_1 \| Q_{11} u_{0} + Q_{12} \G_0(s) u_{0} \|  \\
        &= M_1 \|u_{1}(s)\| \\
        &\ge \| \G_1(s) u_{1}(s) \| \\
        &= \| Q_{21} u_{0} + Q_{22} \G_0(s) u_{0}\| \\
        &\geq \|Q_{21} u_{0}\| - \|Q_{22} \G_0(s) u_{0} \| .
\end{align*}
As the system $\left(\mathcal A, \mathcal B_0, \mathcal C_0 \right)$ is regular with a feedthrough $0$, taking the limit as $s \in \mathbb R$, $s \to \infty$, we get that 
\begin{equation}\label{eqn:M}
      M_1 \| Q_{11} u_{0} \|  \ge \|Q_{21} u_{0}\|, \quad u_0 \in U .
\end{equation}
Since $Q$ is bounded and invertible, the inverse $Q^{-1}$ is bounded as well. Thus $Q$ is bounded from below, i.e. there exists $C>0$ such that for $\begin{psmallmatrix}
    u_{0} \\ y_{0}
\end{psmallmatrix} \in U \times Y$ we have
\begin{equation*}
    \left\|Q \begin{pmatrix}
    u_{0} \\ y_{0}
\end{pmatrix} \right\| \ge C \left\|\begin{pmatrix}
    u_{0} \\ y_{0}
\end{pmatrix} \right\|. 
\end{equation*}
However, by \eqref{eqn:M} for every $u_0 \in U$
\begin{align*}
    \left(1+M_{1}^2 \right) \|Q_{11} u_0\|^2 
    &\ge \|Q_{11} u_0\|^2 + \|Q_{21} u_0\|^2 \\
    &= \left\| Q \begin{pmatrix}
        u_0 \\ 0
    \end{pmatrix} \right\|^2 \\
    &\ge C^2 \left \| \begin{pmatrix}
        u_0 \\ 0 
    \end{pmatrix} \right\|^2  = C^2 \|u_0\|^2 .
\end{align*}
Thus, for $\Tilde{C}=\frac{C^2}{1+M_{1}^{2}}$,
\begin{equation*}
    \| Q_{11} u_0 \|^2 \ge \Tilde{C} \|u_0\|^2 .
\end{equation*}
This implies that $\ker Q_{11} =\{0\}$ and $\operatorname{ran} Q_{11}$ is closed. \\
Assuming $Q_{11}$ is not surjective, there exists $u_1 \in U$ with $u_1 \not= 0$ and $u_1 \in \operatorname{ran} Q_{11}^\perp$. Using equation \eqref{eqn:2.11}, we have that for $s \in \mathbb C^{+}_{\alpha}$ there exists $ u_0(s) \in U $ such that 
\begin{equation}\label{eqn:2.13}
    \begin{pmatrix}
        Q_{11} & Q_{12} \\
        Q_{21} & Q_{22}
    \end{pmatrix}  \begin{pmatrix}
       u_{0}(s) \\
       \G_0(s) u_{0}(s)
    \end{pmatrix}  = \begin{pmatrix}
       u_{1}\\
       \G_1(s) u_{1}
    \end{pmatrix}  .
\end{equation}
Since $Q$ is invertible and $\G_{1}(s)$ is bounded on $\mathbb C^{+}_{\alpha}$, \eqref{eqn:2.13} gives that $u_{0}(s)$ is bounded on $\mathbb C^{+}_{\alpha}$ as well. Now,
\begin{align*}
    \|u_{1}\|^{2}
    &= \displaystyle\lim_{s \in \mathbb R, s \to \infty}  \|u_{1}\|^{2} \\
    &= \displaystyle\lim_{s \in \mathbb R, s \to \infty} \left\langle \begin{pmatrix}
        u_{1} \\ 
        0
    \end{pmatrix}, Q \begin{pmatrix}
        u_{0}(s) \\
        \G_{0}(s) u_{0}(s)
    \end{pmatrix} \right\rangle \\
    &= \displaystyle\lim_{s \in \mathbb R, s \to \infty} \left\langle u_{1} , \left(Q_{11} + Q_{12} \G_{0}(s) \right) u_{0}(s) \right\rangle \\
    &= \displaystyle\lim_{s \in \mathbb R, s \to \infty} \left\langle u_{1} , Q_{12} \G_{0}(s) u_{0}(s) \right\rangle = 0
\end{align*}
which gives a contradiction. We conclude that $Q_{11}$ is invertible. 
\end{proof}

\section{Analysis of well-posedness}\label{section3}
We now turn our attention to the class of systems from \eqref{eqn:system1},
where $\tau : \mathrm H^2((0,1);\mathbb F^n) \to \mathbb F^{4n}$ is the trace operator given by \eqref{eqn:trace}.
We make the following assumptions.
\begin{assumption}\label{assumption}
    \begin{itemize}
        \item $P_1, \, P_2 \in \mathbb F^{n \times n}$ such that $P_1$ is self-adjoint and $P_2$ is skew-adjoint and invertible;
        \item $P_0 \in \mathrm L^\infty((0,1);\mathbb F^{n \times n})$;
        \item $\mathcal H \in C^1([0,1]; \mathbb F^{n\times n})$ such that $\mathcal H(\zeta)$ is self-adjoint for all $\zeta \in [0,1]$ and there exist $M, m > 0$ such that $mI \le \mathcal H(\zeta) \le M I$ for all $\zeta \in [0,1]$;
        \item The matrices $W_{B,1}, W_C \in \mathbb F^{m \times 4n}$ with $0<m\le 2n$, and $W_{B,2}\in \mathbb F^{(2n-m) \times 4n}$ are such that $\left[\begin{smallmatrix} W_{B,1}\\ W_{B,2} \\ W_{C} \end{smallmatrix}\right]$ has full row rank. 
    \end{itemize}
\end{assumption}

We choose $X = \mathrm L^2((0,1);\mathbb F^n)$ equipped with the inner product \begin{equation}\label{inner_product_X}
    \left\langle f, g\right\rangle_{X} = \frac{1}{2} \int_0^1 g(\zeta)^\ast \mathcal H(\zeta) f(\zeta) d\zeta. 
\end{equation} 
We note that the standard $\mathrm{L^2}$-norm is equivalent to the norm defined by \eqref{inner_product_X}, i.e., for all $f \in \mathrm{L^2}((0,1); \mathbb F^n)$ 
$$ \frac{m}{2} \|f\|^2_{\mathrm{L^2}} \leq \|f\|^2_X \leq \frac{M}{2} \|f\|^2_{\mathrm{L^2}} . $$

\begin{remark}\label{rem:3.3}
By defining
\begin{equation}\label{eqn:3.3}
\begin{aligned}
    \mathcal{A} x &= \left( P_2 \frac{\partial^2}{\partial\zeta^2}  + P_1 \frac{\partial}{\partial \zeta} + P_0 \right) \mathcal H x, \\
    \BB x &= W_{B,1} \tau(\mathcal H x), \\
    \CC x &= W_C \tau(\mathcal H x), \\
    D(\mathcal{A}) &= \left\{ x \in \mathrm L^2((0,1);\mathbb F^n) \, | \, \mathcal H x \in \mathrm H^2((0,1);\mathbb F^n), \, W_{B,2} \tau(\mathcal H x)=0 \right\}, \\
     D(\BB) &= D(\A),
\end{aligned}
\end{equation}
system \eqref{eqn:system1} is a boundary control and observation system, provided the operator $A_0 \coloneqq \A_{\mid\ker \BB}$ is the infinitesimal generator of a $C_0$-semigroup on $X$, see \cite{GorZwaMas:2005}. 
\end{remark}

\begin{remark}\label{remark1}
If \eqref{eqn:system1} is an impedance passive boundary control and observation system, then the operator $A_0 = \mathcal{A}_{\mid\ker \BB}$ is dissipative and in this case the $C_0$-semigroup generated by $A_0$ is a contraction semigroup on $X$, see \cite{Aug:16,Villegas:07}. If in addition, $m =2n$ and $P_0(\zeta)^\ast = -P_0(\zeta)$ for a.e. $\zeta \in (0,1)$, then the impedance passivity is equivalent to
\begin{equation*}
     \begin{bmatrix}
    \tilde W_B \Sigma \tilde W_B^\ast  & \tilde W_B \Sigma \tilde W_C^\ast \\
    \tilde W_C \Sigma \tilde W_B^\ast & \tilde W_C \Sigma \tilde W_C^\ast 
    \end{bmatrix}^{-1} \leq \begin{bmatrix}
        0 & I \\
        I & 0 
    \end{bmatrix},
\end{equation*} 
and impedance energy-preservation is equivalent to
\begin{equation*}
     \begin{bmatrix}
    \tilde W_B \Sigma \tilde W_B^\ast  & \tilde W_B \Sigma \tilde W_C^\ast \\
    \tilde W_C \Sigma \tilde W_B^\ast & \tilde W_C \Sigma \tilde W_C^\ast 
    \end{bmatrix}^{-1} = \begin{bmatrix}
        0 & I \\
        I & 0 
    \end{bmatrix},
\end{equation*} 
where 
    \begin{align*}
     \tilde W_B &= \frac{1}{\sqrt{2}} \left[\begin{matrix}
         W_{B,1} \\ W_{B,2}
     \end{matrix}\right] \left[\begin{matrix}
           R & I \\
          -R & I
     \end{matrix} \right], \quad \tilde W_C= \frac{1}{\sqrt{2}} W_C \left[\begin{matrix}
           R & I \\
          -R & I
     \end{matrix} \right] , \\
     \Sigma &=\left[\begin{matrix}  0 & I \\ I & 0 \end{matrix}\right], \quad 
        R= \left[\begin{matrix}
         0 & -P_2^{-1} \\
         P_2^{-1} & P_2^{-1} P_1 P_2^{-1}
     \end{matrix} \right], 
     \end{align*}
see \cite{AugJac:14,Aug:16}.
Furthermore, by Theorem \ref{thm:classical_solution}, for smooth input and initial conditions, system \eqref{eqn:system1} has classical solutions over the time interval $[0,T]$, $T>0$. Additionally, the output $y$ is well-defined and continuous on this interval. Well-posedness, on the other hand, not only guarantees the existence of mild solutions but also ensures that the output function $y$ satisfies $y\in \mathrm L^2((0,T);\mathbb F^m)$ for arbitrary initial conditions $x_0\in X$ and inputs $u \in \mathrm L^2((0,T);\mathbb F^m)$.  
\end{remark}

In what follows, we define 
\begin{equation}\label{def:A_0}
\begin{aligned}
    \mathcal A_{e} x &= \left( P_2 \frac{\partial^2}{\partial\zeta^2}  + P_1 \frac{\partial}{\partial \zeta} + P_0 \right) \mathcal H x, \\
    \mathcal B_{e} x &= \frac{1}{\sqrt{2}} \begin{bmatrix} 
            P_2 \frac{\partial}{\partial \zeta} (\mathcal{H} x)(1) + \tfrac{1}{2} P_1 (\mathcal{H} x)(1) \\  
           P_2 \frac{\partial}{\partial \zeta} (\mathcal{H} x)(0) + \frac{1}{2} P_1 (\mathcal{H} x)(0) 
     \end{bmatrix}, \\
    \mathcal C_{e} x &= \frac{1}{\sqrt{2}} 
    \begin{bmatrix} 
        (\mathcal{H}x)(1) \\  
        -(\mathcal{H}x)(0) 
    \end{bmatrix} , \\
    D(\mathcal A_{e}) &= \left\{ x \in \mathrm L^2((0,1);\mathbb F^n) \, | \, \mathcal H x \in \mathrm H^2((0,1);\mathbb F^n) \right\}, \\
    D(\mathcal B_{e}) &= D(\mathcal A_{e}). \\
\end{aligned}
\end{equation}
We now proceed with the analysis of the well-posedness of system \eqref{eqn:system1}. As an initial step, we state the following theorem. 
\begin{thm}\label{prop0} 
    With the operators in \eqref{def:A_0}, consider the boundary control and observation system 
    \begin{equation}\label{internal_sys}
    \begin{aligned}
        \dot{x}(t) &= \mathcal A_{e} x(t) , \quad x(0)=x_{0}, \\
         u_{e}(t) &= \mathcal B_{e} x(t) , \\
         y_{e}(t) &= \mathcal C_{e} x(t) , \quad t>0
    \end{aligned}
    \end{equation}
  with the state $x$, the input $u_{e}$ and the output $y_{e}$. We assume that the system satisfies Assumption \ref{assumption}. Then, system \eqref{internal_sys} is well-posed, impedance passive and regular. Moreover, for some $\alpha>0$ and $c>0$, the transfer function $\G_{e}$ satisfies 
  \begin{equation*}
    \|\G_{e}(s) \|_{\mathcal L(U)} \leq \frac{c}{\sqrt{\Re(s)}}  , \quad s \in \mathbb C^{+}_{\alpha}.
\end{equation*} 
Moreover, if $P_0(\zeta)^\ast = -P_0(\zeta)$ for a.e. $\zeta \in (0,1)$, then system \eqref{internal_sys} is impedance energy-preserving.
\end{thm} 
\begin{proof}
Using integration by parts, $P_2^\ast=-P_2$, $P_1^\ast = P_1$ and $\mathcal H(\zeta)^\ast=\mathcal H(\zeta)$ for all $\zeta \in [0,1]$, we obtain that for $x \in D(\mathcal A_{e})$ 
\begin{align*} 
\Re\left\langle \A_{e} x, x\right\rangle_X  
    &= \Re \left\langle P_2 (\mathcal H x)'' + P_1 (\mathcal H x)' + P_0 \mathcal H x, x\right\rangle_X \\
    &= \frac{1}{2} \Re \left\langle P_2 (\mathcal H x)'', \mathcal H x \right\rangle_{\mathrm{L}^2} + \frac{1}{2}  \Re \left\langle P_1 (\mathcal H x)', \mathcal H x \right\rangle_{\mathrm{L}^2} + \frac{1}{2} \left\langle P_0 \mathcal H x, \mathcal H x \right\rangle_{\mathrm L^2} .
\end{align*}

By integration by parts and using that $P_{2}$ is skew-adjoint, we have 
\begin{align*}
    2 \Re \left\langle P_2 (\mathcal H x)'', \mathcal H x \right\rangle_{\mathrm{L}^2}
    &= \int_{0}^{1} (\mathcal H x)(\zeta)^\ast P_{2} (\mathcal H x)''(\zeta) d\zeta  +  \int_{0}^{1} \left(P_{2} (\mathcal H x)''(\zeta)\right)^\ast (\mathcal H x)(\zeta) d\zeta \\
    &= 2\Re \Big[ (\mathcal H x)(\zeta)^\ast P_2 (\mathcal H x)'(\zeta) \Big]_{0}^{1} - \int_{0}^{1} (\mathcal H x)'(\zeta)^\ast P_{2} (\mathcal H x)'(\zeta) d\zeta \\
    &\qquad - \int_{0}^{1} \left( P_{2}(\mathcal H x)'(\zeta)\right)^\ast (\mathcal H x)'(\zeta) d\zeta \\
    &= 2\Re \Big[ (\mathcal H x)(\zeta)^\ast P_2 (\mathcal H x)'(\zeta) \Big]_{0}^{1} - \int_{0}^{1} (\mathcal H x)'(\zeta)^\ast P_{2} (\mathcal H x)'(\zeta) d\zeta \\
    &\qquad + \int_{0}^{1} (\mathcal H x)'(\zeta)^\ast  P_{2}(\mathcal H x)'(\zeta) d\zeta \\
    &= 2\Re \Big[ (\mathcal H x)(\zeta)^\ast P_2 (\mathcal H x)'(\zeta) \Big]_{0}^{1} .
\end{align*}
Similarly, using that $P_{1}$ is self-adjoint we have
\begin{align*}
    \Re \left\langle P_1 (\mathcal H x)', \mathcal H x \right\rangle_{\mathrm{L}^2}
    &= \Re \Big[ (\mathcal H x)(\zeta)^\ast P_{1} (\mathcal H x)(\zeta) \Big]_{0}^{1} - \Re\Big[\int_{0}^{1} (\mathcal H x)'(\zeta)^\ast P_{1} (\mathcal H x)'(\zeta) d\zeta \Big] \\
    &= \Re \Big[ (\mathcal H x)(\zeta)^\ast P_{1} (\mathcal H x)(\zeta) \Big]_{0}^{1} - \Re \left\langle P_1 (\mathcal H x)', \mathcal H x \right\rangle_{\mathrm{L}^2} .
\end{align*}
Thus,
\begin{equation*}
    \Re \left\langle P_1 (\mathcal H x)', \mathcal H x \right\rangle_{\mathrm{L}^2} = \frac{1}{2} \Re \Big[ (\mathcal H x)(\zeta)^\ast P_{1} (\mathcal H x)(\zeta) \Big]_{0}^{1} ,
\end{equation*}
and then 
\begin{align*}
     \Re\left\langle \A_{e} x, x\right\rangle_X
    &= \frac{1}{2} \Re \left[ (\mathcal H x)(\zeta)^\ast \left(P_2 (\mathcal H x)'(\zeta) + \frac{1}{2} P_1 (\mathcal H x)(\zeta) \right) \right]_0^1 + \frac{1}{2} \left\langle P_0 \mathcal H x, \mathcal H x \right\rangle_{\mathrm L^2} \\
    &= \Re \left\langle u_{e}, y_{e} \right\rangle_{\mathbb F^{2n}} + \frac{1}{2} \left\langle P_0 \mathcal H x, \mathcal H x \right\rangle_{\mathrm L^2} .
\end{align*}
We see that if $P_0(\zeta)^\ast = -P_0(\zeta)$ for a.e. $\zeta \in (0,1)$, then $\left\langle P_0 \mathcal H x, \mathcal H x \right\rangle_{\mathrm L^2} =0$. Thus, the system is impedance energy-preserving. Nonetheless, as $P_0 \mathcal H$, seen as a multiplication operator on $X$, is a bounded operator on the state space $X$, the well-posedness and regularity remains unchanged under bounded perturbations, see Lemma \ref{lem:perturbation}. Moreover, the feedthrough term will remain the same. Thus, without loss of generality, we assume that $P_0=0$. Therefore, the system is impedance passive. By Remark \ref{remark1}, $A_{e} \coloneqq\mathcal A_{e\mid \ker \BB_0}$ generates a contraction semigroup on $X$. Next, we show that the system is well-posed. Since $A_{e}$ generates a contraction semigroup, by Theorem \ref{thm:well-posedness} it remains to prove that the transfer function of the system is bounded on some vertical line in the open right half-plane. As both $u_{e}$ and $y_{e}$ in \eqref{def:A_0} contain the common factor $\frac{1}{\sqrt{2}}$, we may, without loss of generality, remove this factor, as boundedness is preserved under multiplication by a nonzero constant. \\
Let $s=r+i\omega \in \mathbb{C}$, $r>1$ and $u_{e,0} \in \mathbb F^{2n}$. By $x_0$ we denote the solution of the ordinary differential equation
\begin{equation}\label{ODE}
  s x_0(\zeta) =  P_2 (\mathcal{H} x_0)''(\zeta) +  P_1 (\mathcal{H} x_0)'(\zeta) ,
\end{equation}
with the input $u_{e,0}$ 
\begin{equation}\label{u_e,0}
 u_{e,0} = \begin{bmatrix} 
  P_2 (\mathcal{H} x_0)'(1) + \frac{1}{2} P_1 (\mathcal{H} x_0)(1) \\  
 P_2 (\mathcal{H} x_0)'(0) + \frac{1}{2} P_1 (\mathcal{H} x_0)(0)  
 \end{bmatrix} = \begin{bmatrix}
     u_{e,0}^{1} \\
     u_{e,0}^{2}
 \end{bmatrix} , 
\end{equation}
and we define the output $y_{e,0}$ by
\begin{equation}\label{y_e,0}
  y_{e,0} \coloneqq \begin{bmatrix} 
  (\mathcal{H} x_0)(1) \\  
  -(\mathcal{H} x_0)(0) 
  \end{bmatrix} = \begin{bmatrix}
     y_{e,0}^{1} \\
     y_{e,0}^{2}
 \end{bmatrix} .
\end{equation}
Since for every $\zeta \in [0,1]$, $mI \le \mathcal{H}(\zeta) \le MI $, the norms $\|\mathcal H(\zeta) \cdot\|$, $\|\mathcal H^{\frac{1}{2}}(\zeta)\cdot\|$ and the Euclidean norm on $\mathbb F^{n}$ are equivalent and the equivalence constant can be chosen independent of $\zeta$. Using this and the Fundamental Theorem of Calculus, we have
\begin{align*}
\|y_{e,0}\|^{2} 
 &=  \|(\mathcal{H}x_0)(1) \|^2 + \|(\mathcal{H}x_0)(0)\|^2 \\
   &\lesssim x_0(1)^\ast \mathcal{H}(1) x_0(1) + x_0(0)^\ast \mathcal{H}(0) x_0(0) \\
  &=\ \int_0^1 \frac{d}{d\zeta} \left( (2\zeta-1) x_0(\zeta)^\ast \mathcal{H}(\zeta) x_0(\zeta) \right) d\zeta  \\ 
  &=\ 2\int_0^1 x_0(\zeta)^\ast \mathcal{H}(\zeta) x_0(\zeta) d\zeta  + \int_0^1 (2\zeta-1) \left((\mathcal{H}x_0)'(\zeta)\right)^\ast x_0(\zeta) d\zeta \\
  &\qquad \quad + \int_0^1 (2\zeta-1) (\mathcal{H}x_0)(\zeta)^\ast x_0(\zeta)'  d\zeta \\
  &=\ 4\|x_0\|_X^2  + \int_0^1 (2\zeta-1) \left[ \left( (\mathcal{H}x_0)'(\zeta)\right)^\ast x_0(\zeta) + x_0(\zeta)^\ast (\mathcal{H}x_0)'(\zeta) \right] d\zeta \\
  & \qquad \quad -\int_0^1 (2\zeta-1) x_0(\zeta)^\ast \mathcal{H}(\zeta)' x_0(\zeta) d\zeta \\
  &=\ 4\|x_0\|_X^2 - \int_0^1 (2\zeta-1) x_0(\zeta)^\ast \mathcal{H}(\zeta)' x_0(\zeta) d\zeta  \\
  &\qquad \quad + 2 \int_0^1 (2\zeta-1) \Re  \left[ x_0(\zeta)^\ast (\mathcal{H}x_0)'(\zeta)\right] d\zeta .
\end{align*}
Since $\mathcal H \in C^1([0,1];\mathbb F^n)$, we obtain
$$ \int_0^1 x_0(\zeta)^\ast \mathcal{H}(\zeta)' x_0(\zeta) d\zeta \lesssim \|x_{0}\|^{2}_{\mathrm L^{2}} . $$
Hence,
\begin{equation}\label{eqn:lem1}
    \|y_{e,0}\|^2 
    \lesssim \|x_0\|_X^2  + 2\int_0^1 (2\zeta-1) \Re  \left[ x_0(\zeta)^\ast (\mathcal{H}x_0)'(\zeta)\right]  d\zeta.
\end{equation}
For simplicity, we call $h_0=\mathcal{H} x_0$. Using equation \eqref{ODE}, we calculate 
\begin{align*}
    h_0'(\zeta)^\ast & x_0(\zeta) + x_0(\zeta)^\ast h_0'(\zeta) \\
     &= \frac{1}{s} h_0'(\zeta)^\ast \left( P_2 h_0''(\zeta) + P_1 h_0'(\zeta) \right) + \frac{1}{\overline{s}}\left( P_2h_0''(\zeta) + P_1 h_0'(\zeta) \right)^\ast h_0'(\zeta)  \\
    &= \frac{\overline{s}}{|s|^2} h_0'(\zeta)^\ast \left( P_2 h_0''(\zeta) + P_1 h_0'(\zeta) \right) +  \frac{s}{|s|^2}\left( P_2h_0''(\zeta) + P_1 h_0'(\zeta) \right)^\ast h_0'(\zeta)  \\
    &= \frac{r}{|s|^2} \left[h_0'(\zeta)^\ast \left( P_2 h_0''(\zeta) + P_1 h_0'(\zeta) \right) + \left( P_2h_0''(\zeta) + P_1 h_0'(\zeta) \right)^\ast h_0'(\zeta) \right] \\ 
    & \qquad - \frac{i\omega}{|s|^2} \left[h_0'(\zeta)^\ast \left( P_2 h_0''(\zeta) + P_1 h_0'(\zeta) \right) -  \left( P_2h_0''(\zeta) + P_1 h_0'(\zeta) \right)^\ast h_0'(\zeta) \right]  \\
    &= \frac{r}{|s|^2} \left[h_0'(\zeta)^\ast s x_0(\zeta) + \overline{s} x_0(\zeta)^\ast h_0'(\zeta) \right]  - \frac{2i\omega}{|s|^2} \Re \left[h_0'(\zeta)^\ast P_2 h_0''(\zeta) \right]  \\  
     &= \frac{2r}{|s|^2} \Re \left[ s h_0'(\zeta)^\ast x_0(\zeta) \right] -\frac{i\omega}{|s|^2} \left( \left (h_0'(\zeta)\right)^\ast P_2 h_0'(\zeta) \right)' .
\end{align*}
Substituting the previous expression into \eqref{eqn:lem1} and using the Cauchy-Schwarz inequality, the inequality \eqref{eqn:lem1} becomes
\begin{align*}
\|y_{e,0}\|^{2} 
& \lesssim \|x_0\|_X^2  + 2\int_0^1 (2\zeta-1) \Re  \left[ x_0(\zeta)^\ast h_0'(\zeta)\right]  d\zeta\\
 &  =  \|x_0\|_X^2  +  \frac{2r}{|s|^2} \int_0^1 (2\zeta-1) \Re \left[ s h_0'(\zeta)^\ast x_0(\zeta) \right]d\zeta \\
 &\qquad -\frac{i\omega}{|s|^2}  \int_0^1 (2\zeta-1) \left[ h_0'(\zeta)^\ast P_2  h_0'(\zeta) \right]' d\zeta  \\
    &\lesssim  \|x_0\|_X^2  +\frac{r|s|}{|s|^2} \int_0^1 \| h_0'(\zeta)\| \|x_0(\zeta)\| d\zeta - \frac{i\omega}{|s|^2}  \int_0^1 (2\zeta-1) \left[ h_0'(\zeta)^\ast P_2  h_0'(\zeta) \right]' d\zeta \\
    &\lesssim  \|x_0\|_X^2  +\frac{r}{|s|}\| h_0'\|_\mathrm{L^2} \|x_0\|_{\mathrm{L^2}} -\frac{i\omega}{|s|^2} \left[(2\zeta-1) h_0'(\zeta)^\ast P_2  h_0'(\zeta) \right]_0^1 \\ 
    &\qquad 
    +\frac{2i\omega}{|s|^2} \int_0^1 h_0'(\zeta)^\ast P_2  h_0'(\zeta) d\zeta .
\end{align*}
Using the Cauchy-Schwarz inequality and the fact that the norms $\|\cdot\|_{\mathrm L^2}$ and $\|\cdot\|_X$ are equivalent, we obtain that
\begin{equation}\label{eqn:y1}
    \|y_{e,0}\|^{2} 
    \lesssim \|x_0\|_X^2 +\frac{r}{|s|}\| h_0'\|_{\mathrm{L^2}} \|x_0\|_{X} -\frac{i\omega}{|s|^2} \Big[(2\zeta-1) h_0'(\zeta)^\ast P_2  h_0'(\zeta) \Big]_0^1 +\frac{1}{|s|}\| h_0'\|_{\mathrm{L^2}}^2 .
\end{equation}
On the other hand, using the equalities \eqref{u_e,0} and \eqref{y_e,0}, together with the invertibility of the matrix $P_2$, we have that 
\begin{align*}
    \Big[ (2 &\zeta -1) h_0'(\zeta)^\ast P_2 h_0'(\zeta) \Big]_0^1 \\
    &= h_0'(1)^\ast P_2 h_0'(1) + h_0'(0)^\ast P_2 h_0'(0) \\
    &= \left[ \frac{1}{2} P_1 h_0(1) - P_2 h_0'(1) - \frac{1}{2} P_1 h_0(1) \right]^\ast P_2^{-1} \left[ \frac{1}{2} P_1 h_0(1) + P_2 h_0'(1) - \frac{1}{2} P_1 h_0(1) \right] \\
    &\quad + \left[ \frac{1}{2} P_1 h_0(0) - P_2 h_0'(0) - \frac{1}{2} P_1 h_0(0) \right]^\ast P_2^{-1} \left[ \frac{1}{2} P_1 h_0(0) + P_2 h_0'(0) - \frac{1}{2} P_1 h_0(0) \right] \\
    &= \begin{bmatrix}
        -u_{e,0}^{1} + P_1 y_{e,0}^{1} \\
        - u_{e,0}^{2} - P_1 y_{e,0}^{2}
    \end{bmatrix}^\ast \begin{bmatrix}
        P_2^{-1} & 0 \\
        0 & P_2^{-1}
    \end{bmatrix} \begin{bmatrix}
        u_{e,0}^{1} - P_1 y_{e,0}^{1} \\
        u_{e,0}^{2} + P_1 y_{e,0}^{2}
    \end{bmatrix},
\end{align*}
Hence,
\begin{equation*}
    \left\| \Big[ (2\zeta-1) h_0'(\zeta)^\ast P_2 h_0'(\zeta) \Big]_0^1 \right\| 
    \lesssim \|u_{e,0}\|^2 + \|u_{e,0}\| \|y_{e,0}\| + \|y_{e,0}\|^2 .
\end{equation*}
Thus, the inequality \eqref{eqn:y1} becomes
\begin{equation}\label{eqn:y2}
\begin{aligned}
    \|y_{e,0}\|^{2} 
    &\lesssim \|x_0\|_X^2  +\frac{r}{|s|}\| h_0'\|_{\mathrm{L^2}} \|x_0\|_{X} + \frac{1}{|s|} \|u_{e,0}\|^2 + \frac{1}{|s|} \|u_{e,0}\| \|y_{e,0}\| \\
    & \qquad  + \frac{1}{|s|}\|y_{e,0}\|^2 +\frac{1}{|s|}\| h_0'\|_{\mathrm{L^2}}^2 .  
\end{aligned}
\end{equation}
In order to get an estimate of the transfer function $\G_{e}(s)$, we need to evaluate the norm $\| h_0'\|_{\mathrm{L^2}}$. To this end, we calculate
\begin{align*}
\bigg[ \bigg(P_2 & h_0'(\zeta) + \frac{1}{2} P_1 h_0(\zeta) \bigg)^\ast P_2 h_0(\zeta) \bigg]' \\
  & = \left( P_2  h_0''(\zeta) + \frac{1}{2} P_1  h_0'(\zeta) \right)^\ast P_2  h_0(\zeta) + \left( P_2  h_0'(\zeta) + \frac{1}{2} P_1  h_0(\zeta) \right)^\ast P_2 h_0'(\zeta). 
 \intertext{Plugging in the first term of equation the equation \eqref{ODE}, we get} 
  &= \left( s x_0(\zeta) - \frac{1}{2} P_1  h_0'(\zeta) \right)^\ast P_2  h_0(\zeta) + \|P_2 h_0'(\zeta)\|^2 + \frac{1}{2} h_0(\zeta)^\ast P_1 P_2 h_0'(\zeta) \\
  &= \overline{s} x_0(\zeta)^\ast P_2 h_0(\zeta) - \frac{1}{2} h_0'(\zeta)^\ast P_1 P_2 h_0(\zeta) - \frac{1}{2} \overline{h_0'(\zeta)^\ast P_2 P_1 h_0(\zeta)} + \|P_2 h_0'(\zeta)\|^2 .
\end{align*}
This implies
\begin{equation}\label{eqn:1}
\begin{aligned}
2 \Re &\left[ \left(P_2  h_0'(\zeta) + \frac{1}{2} P_1  h_0(\zeta) \right)^\ast P_2 h_0(\zeta)\right]' \\
     &= 2 \Re \left[\overline{s} x_0(\zeta)^\ast P_2 h_0(\zeta)\right] + 2 \|P_2 h_0'(\zeta)\|^2  - h_0'(\zeta)^\ast \left(P_1 P_2 + P_2 P_1\right) h_0(\zeta) . 
\end{aligned}
\end{equation}
Since $P_2$ is invertible, we have that the norm $\|P_2 \cdot\|$ is equivalent to the Euclidean norm $\|\cdot \|$ on $\mathbb F^{n}$. Therefore, with \eqref{eqn:1} we have that
\begin{align*}
\int_0^1 \| h_0'(\zeta) \|^2 d\zeta 
    &\lesssim \int_0^1 \| P_2 h_0'(\zeta)\|^2 d\zeta \\
    &=  \Re  \left[\left(P_2  h_0'(\zeta) + \frac{1}{2} P_1  h_0(\zeta) \right)^\ast P_2 h_0(\zeta) \right]_0^1 \\
    &\qquad \quad - \int_0^1 \mathrm{Re}\Big[\overline{s} x_0(\zeta)^\ast P_2 \mathcal H(\zeta) x_0(\zeta)\Big] d\zeta \\
    &\qquad \quad + \frac{1}{2} \int_0^1 h_0(\zeta)^\ast \left(P_1 P_2 + P_2 P_1 \right) h_0'(\zeta) d\zeta .
\end{align*}
By virtue of $\mathcal H \in C^1([0,1];\mathbb F^{n\times n})$, equation \eqref{u_e,0}, and the Cauchy-Schwarz inequality, we obtain that
\begin{equation*}
     \| h_0' \|_{\mathrm L^2} 
    \lesssim\  |s|\|x_0\|_{X}^2  + \|x_0\|_{X} \| h_0'\|_{\mathrm{L^2}} + \|u_{e,0}\| \|y_{e,0}\| .
\end{equation*}
Therefore, we obtain that for some constant $c>0$ 
\begin{equation*}
    \| h_0'\|_{\mathrm{L}^2}^2 - c\|x_0\|_X \| h_0'\|_{\mathrm{L}^2} - c|s|\|x_0\|^2_X - c\|u_{e,0}\| \|y_{e,0}\| \leq 0 .
\end{equation*}
This implies
\begin{equation}\label{eqn:lem2}
     \| h_0'\|_{\mathrm{L}^2} 
    \leq c\|x_0\|_X + \sqrt{c^2\|x_0\|_X^2 + 4c \left(|s| \|x_0\|_X^2 + \|u_{e,0}\| \|y_{e,0}\| \right)} .
\end{equation}
Since the system is impedance passive, there holds by Corollary \ref{cor:passive}
$$ \|x_0\|_X^2 \lesssim \frac{1}{r}  \|u_{e,0}\|  \|y_{e,0}\|. $$
Thus, for $|s|>1$, the inequality \eqref{eqn:lem2} becomes 
\begin{align*}
    \| h_0'\|_{\mathrm{L}^2} 
   &\lesssim \frac{1}{\sqrt{r}}\sqrt{\|u_{e,0}\|} \sqrt{\|y_{e,0}\|}   + \sqrt{ \frac{ |s|}{r}\|u_{e,0}\| \|y_{e,0}\| +\|u_{e,0}\| \|y_{e,0}\|  } \\
    &\lesssim \frac{\sqrt{|s|}}{\sqrt{r}} \sqrt{\|u_{e,0}\|} \sqrt{\|y_{e,0}\|}. 
\end{align*}
From inequality \eqref{eqn:y2}, we recall that
\begin{align*}
    \|y_{e,0}\|^2 
    &\lesssim  \|x_0\|_X^2  +\frac{r}{|s|}\| h_0'\|_{\mathrm{L^2}} \|x_0\|_{X} + \frac{1}{|s|} \|u_{e,0}\|^2 + \frac{1}{|s|} \|u_{e,0}\| \|y_{e,0}\| \\
    & \qquad \quad + \frac{1}{|s|}\|y_{e,0}\|^2 +\frac{1}{|s|}\| h_0'\|_{\mathrm{L^2}}^2 .
\end{align*}
Using again the impedance passivity of the system and the earlier estimate of $ \| h_0'\|_{\mathrm{L}^2}$, we deduce that
\begin{align*} 
\|y_{e,0}\|^2
    &\lesssim \frac{1}{r} \|y_{e,0}\| \|u_{e,0}\| +\frac{r}{|s|} \frac{1}{\sqrt{r}} \frac{\sqrt{|s|}}{\sqrt{r}} \|u_{e,0}\|\|y_{e,0}\| +\frac{1}{r} \|u_{e,0}\|^2 + \frac{1}{r}\|y_{e,0}\|^2 \\
    & \qquad \quad +\frac{1}{|s|} \frac{|s|}{r} \|u_{e,0}\| \|y_{e,0}\|  \\
     &\lesssim \left(\frac{2}{r} + \frac{1}{\sqrt{|s|}} \right) \|u_{e,0}\| \|y_{e,0}\| + \frac{1}{r} \|u_{e,0}\|^2 + \frac{1}{r}\|y_{e,0}\|^2 .
\end{align*}
Therefore, for some constant $c>0$ (independent of $r$) we have
$$ \left(1-\frac{c}{r}\right) \|y_{e,0}\|^2 - \left(\frac{2c}{r} + \frac{c}{\sqrt{|s|}} \right) \|u_{e,0}\| \|y_{e,0}\| - \frac{c}{r} \|u_{e,0}\|^2 \leq 0. $$
For $r>c$, this implies
\begin{equation*}
    \|y_{e,0}\| 
    \leq\ \frac{1}{2\left(1 -\frac{c}{r} \right)} \sqrt{c^2\left(\frac{2}{r}+\frac{1}{\sqrt{|s|}}\right)^2 + 4 \frac{c}{r}\left(1 -\frac{c}{r} \right)} \|u_{e,0}\| + \frac{ \left( \frac{2}{r}+\frac{1}{\sqrt{|s|}} \right)c }{2(1 -\frac{c}{r})} \|u_{e,0}\|.
\end{equation*}
Therefore, for $r \gg 1$, we obtain that
\begin{equation}
     \|y_{e,0}\| \lesssim\ \frac{1}{\sqrt{r}} \|u_{e,0}\|.
\end{equation}
We have shown that for $\Re(s)$ sufficiently large the transfer function satisfies the inequality 
\begin{equation}\label{eqn:transfer_fct_inequality}
    \|\G_{e}(s) \|_{\mathcal L(U)} \leq \frac{c}{\sqrt{r}}  , \quad c>0 
\end{equation} 
and therefore, it is bounded on a vertical line in the open right half-plane $ \mathbb C^{+}_{r_0}$
for some $r_0>0$. By Theorem \ref{thm:well-posedness}, system \eqref{internal_sys} is well-posed. Furthermore, by \eqref{eqn:transfer_fct_inequality} we obtain that
\begin{equation*}
    \lim_{\Re(s) \to \infty} \G_{e}(s) = 0,
\end{equation*}
which  implies that the system is regular and its feedthrough term  is zero.
\end{proof}

\medskip
Using the input $u_{e}$ and the output $y_{e}$ defined in Theorem \ref{prop0}, see \eqref{def:A_0} and \eqref{internal_sys}, we may rewrite system \eqref{eqn:system1}. We note that $\begin{bsmallmatrix}
    u_{e} \\ y_{e}
\end{bsmallmatrix}$ is obtained from $\tau(\mathcal{H}x)$ through an invertible linear transformation.
Hence $\begin{bsmallmatrix}
    u_{e} \\ y_{e}
\end{bsmallmatrix}$ span the same range as $\tau(\mathcal{H}x)$. Therefore, $u=W_{B,1}\tau(\mathcal{H}x)$, $0= W_{B,2}\tau(\mathcal{H}x)$ and $y= W_{C}\tau(\mathcal{H}x)$ can equivalently be expressed as linear combinations of $u_{e}$ and $y_{e}$, and there exist uniquely determined matrices $K_1$, $K_2 \in \mathbb F^{(2n-m) \times 2n}$ and $B_1$, $B_2$, $C_1$, $C_2\in \mathbb F^{m\times 2n}$ such that the system \eqref{eqn:system1} can equivalently be described by 
\begin{align}
     \frac{\partial x}{\partial t}(\zeta,t) &= \left( P_2 \frac{\partial^2}{\partial \zeta^2} +  P_1 \frac{\partial}{\partial \zeta} + P_0(\zeta) \right)\mathcal{H}(\zeta) x(\zeta,t), \quad t>0, \quad \zeta \in [0,1] \\
     0 &= K_1 u_{e}(t) + K_2 y_{e}(t), \label{input0} \\
     u(t) &= B_1 u_{e}(t) + B_2 y_{e}(t), \label{input} \\
     y(t) &= C_1  u_{e}(t) + C_2  y_{e}(t). \label{output}
\end{align}

The following theorem provides an equivalent condition for the well-posedness of the system \eqref{eqn:system1} in the case $m=2n$. Note that if $m = 2n$, then $W_{B,2}=0$ and thus $K_1=K_2=0$. Recall that we work under Assumption \ref{assumption}.

\begin{thm}\label{0main_result}
    Let $m=2n$. Then, the boundary control and observation system \eqref{eqn:system1} is well-posed if and only if the matrix $B_1$ in \eqref{input} is invertible. In this case, the system is regular and its feedthrough term is $C_{1} B_{1}^{-1}$. 
\end{thm}
\begin{proof}
If $B_1 \in \mathbb F^{2n \times 2n}$ is invertible, then equation \eqref{input} may be reformulated as 
\begin{equation*}
    u_{e}(t) = B_1^{-1}  u(t) - B_1^{-1} B_2 y_{e}(t).
\end{equation*} 
Thus, system \eqref{eqn:system1} is the system \eqref{internal_sys} closed via the output feedback defined by $-B_1^{-1} B_2$. 
From Theorem \ref{prop0}, we have $\displaystyle\lim_{\Re\, s \to \infty} \G_{e}(s)=0$, then, for $\Re(s)$ large enough we get that $I + B_1^{-1} B_2 \G_{e}(s)$ is invertible and its inverse exists and is bounded on a right half-plane. Therefore, the matrix $B_1^{-1} B_2$ defines an admissible feedback operator and thus by Theorem \ref{thm:feedback}, the closed loop system is a well-posed boundary control and observation system with a transfer function is given by 
\begin{equation*}
    \G(s) = \left( C_1 + C_2 \G_{e}(s) \right) B_{1}^{-1} \left( I + B_1^{-1} B_2 \G_{e}(s) \right)^{-1} , \quad s \in \mathbb C^{+}_{\alpha} 
\end{equation*}
for some $\alpha>0$. Furthermore, we see that it converges to $C_1 B_{1}^{-1}$, as $\Re(s) \to \infty$.

Conversely, we assume that the system \eqref{eqn:system1} with $m=2n$ is well-posed and we show that $B_1$ is invertible. By assumption, the matrix
    $\begin{bsmallmatrix} 
    W_{B,1} \\ 
    W_C 
    \end{bsmallmatrix}$ is invertible. Hence, we conclude that 
    $\begin{bsmallmatrix} 
    B_1 & B_2 \\ 
    C_1 & C_2 
    \end{bsmallmatrix}$ 
    is invertible as well. 
As the system \eqref{internal_sys} is well-posed with a feedthrough $0$, we apply Lemma \ref{lem:invertibility} and obtain that $B_1$ is invertible. 
\end{proof}

\medskip
\medskip
If $0 < m \leq 2n$, the sufficient condition of Theorem \ref{0main_result} still holds, which will be shown in the following theorem. Further, we provide a necessary condition.
\begin{cor}\label{corollary1}
   The boundary control and observation system \eqref{eqn:system1} is well-posed and regular if the matrix 
   \begin{equation*}
       \begin{bmatrix}
        K_1 \\ 
        B_1
    \end{bmatrix} 
    = \begin{bmatrix}
        W_{B,1} \\
        W_{B,2}
    \end{bmatrix} P_{2}^{-1} \begin{bmatrix}
        0 & 0 \\
        I & 0 \\
        0 & 0 \\
        0 & I
    \end{bmatrix}
   \end{equation*} 
   in \eqref{input} is invertible. 
\end{cor}
\begin{proof}
Let $\G(s)$ be the transfer function of the boundary control and observation system \eqref{eqn:system1} and suppose that the matrix $\begin{bsmallmatrix}
        K_1 \\ B_1
    \end{bsmallmatrix}$
in \eqref{input} is invertible.
By introducing an additional input $v(t)=W_{B,2} \mathcal H \tau(x)$ and using Theorem \ref{0main_result}, we obtain that the boundary control and observation system with the extended input $\tilde{u} :=\begin{bsmallmatrix}
        u \\ v
    \end{bsmallmatrix}$  
and the extended output $\tilde{y} :=\begin{bsmallmatrix}
    y \\ 
    \omega
\end{bsmallmatrix}$ 
is well-posed and regular. Moreover, its transfer function $\tilde{\G}(s)$ is given by
\begin{equation*}
    \tilde{\G}(s) \tilde{u} = \begin{bmatrix}
        \G(s) & \tilde{\G}_{12}(s) \\
         \tilde{\G}_{21}(s) & \tilde{\G}_{22}(s)
    \end{bmatrix} .
\end{equation*}
In particular, $\G(s)$ has a limit as $\Re(s) \to\infty$. Hence, by choosing $v \equiv 0$, we conclude that 
the boundary control and observation system \eqref{eqn:system1} is well-posed and regular if the matrix $\begin{bsmallmatrix}
        K_1 \\ B_1
    \end{bsmallmatrix}$
in \eqref{input} is invertible.
\end{proof}

\begin{remark}
    The matrix $\begin{bsmallmatrix}
        K_{1} \\ B_{1}
    \end{bsmallmatrix}$ is invertible if and only if the matrix $\begin{bsmallmatrix}
        W_{B,1} \\
        W_{B,2}
    \end{bsmallmatrix} \begin{bsmallmatrix}
        0 & 0 \\
        I & 0 \\
        0 & 0 \\
        0 & I
    \end{bsmallmatrix}$ is invertible, or equivalently, the matrix $\begin{bsmallmatrix}
        Q_{12} & Q_{14} \\
        Q_{22} & Q_{24}
    \end{bsmallmatrix}$ is invertible, where 
    \begin{equation*}
    \begin{bmatrix}
        W_{B,1} \\
        W_{B,2}
    \end{bmatrix} = \begin{bmatrix}
        Q_{11} & Q_{12} & Q_{13} & Q_{14} \\
        Q_{21} & Q_{22} & Q_{23} & Q_{24}
    \end{bmatrix} .
    \end{equation*}
    Note that the square matrices $Q_{12}$, $Q_{14}$, $Q_{22}$, $Q_{24}$ are the matrices that act only on $(\mathcal{H} x)'(1,t)$ and $(\mathcal{H} x)'(0,t)$.
\end{remark}

\section{Exact controllability and exact observability}\label{section4}
After characterizing the well-posedness of the class of boundary control and observation systems \eqref{eqn:system1} in the previous section, we now turn our attention to the property of exact controllability. Throughout this section we assume that $P_0(\zeta)^\ast = -P_0(\zeta)$ for a.e. $\zeta \in (0,1)$ and $m=2n$. Under this assumption, the system takes the form
\begin{equation}\label{eqn:system2}
\begin{aligned}
     \frac{\partial x}{\partial t}(\zeta,t) &= \left( P_2 \frac{\partial^2}{\partial \zeta^2} +  P_1 \frac{\partial}{\partial \zeta} + P_0(\zeta) \right)\mathcal{H}(\zeta) x(\zeta,t), \quad t>0, \quad \zeta \in [0,1] \\
     u(t) &= W_{B,1} \tau(\mathcal H x)(t) , \quad t>0\\
     y(t) &= W_C \tau (\mathcal H x)(t), \quad t>0 \\
    x(\zeta,0) &= x_0(\zeta), \quad \zeta \in [0,1].
\end{aligned}
\end{equation}
From this point onward, we define
\begin{equation*}
\begin{aligned}
    \mathcal{A} x &= \left( P_2 \frac{\partial^2}{\partial\zeta^2}  + P_1 \frac{\partial}{\partial \zeta} + P_0 \right) \mathcal H x, \\
    \BB x &= W_{B,1} \tau(\mathcal H x), \\
    \CC x &= W_{C} \tau(\mathcal H x), \\
    D(\mathcal{A}) &= \left\{ x \in \mathrm L^2((0,1);\mathbb F^n) \, | \, \mathcal H x \in \mathrm H^2((0,1);\mathbb F^n) \right\}, \\
     D(\BB) &= D(\A),
\end{aligned}
\end{equation*}
and we assume that the system \eqref{eqn:system2} is a well-posed boundary control and observation system and we aim to characterize the exact controllability and exact observability in finite time.

\begin{defn}\cite[Definition 6.2.1 and Definition 6.2.12]{CurtZwa:20}
The boundary control and observation system \eqref{BCS} is called exactly controllable (in finite time), if there exists a time $T>0$ such that for every $x_1 \in X$ there exists a control function 
    $u \in \mathrm{L^2}((0,T);\mathbb F^m)$ such that the corresponding mild solution $x$ with $x_0=0$ satisfies
   $$ x(0) = 0, \quad x(T) = x_1. $$
It is called exactly observable (in finite time) if there exists a time $T>0$ and a constant $m>0$ such that for every initial condition $x_0\in X$ the corresponding mild solution $(x,y)$ (with $u=0$) satisfies
\begin{equation*}
    \int_0^T \|y(t)\|^2 dt \ge m \|x_0\|^2.
\end{equation*}
\end{defn}

We start by providing the definition of optimizability.
\begin{defn}\cite[Definition 1.1]{Weiss_Rebarber:97}
    A boundary control and observation system  is called optimizable if for every initial condition $x_0 \in X$, there exists an input function $u \in \mathrm{L^2}((0,\infty);\mathbb F^m)$ such that the mild solution $x$ satisfies
    $$ \int_0^\infty \|x(t)\|^2 dt < \infty . $$
\end{defn}
Note that exact controllability implies optimizability and in \cite{Weiss_Rebarber:97}, a sufficient condition for exact controllability is presented in terms of optimizability.

\begin{prop}\label{prop:optimizability}\cite[Corollary 2.2]{Weiss_Rebarber:97}
    If the boundary control and observation system
    \eqref{BCS} is optimizable and $-A_0\coloneqq -\mathcal{A}_{\mid\ker \BB}$ generates a bounded $C_0$-semigroup, then system \eqref{BCS} is exactly controllable.
\end{prop}

\medskip
 
Next, we show that if the boundary control and observation system \eqref{eqn:system2} conserves energy over time, then the system is naturally exactly controllable and exactly observable. Before establishing this first we show that impedance energy-preservation is invariant under duality. Therefore, in the light of duality concept between controllability and observability \cite{tucsnak2009observation}, it suffices to prove exact controllability, as exact observability follows by duality. To this end, we state the following lemma.

\begin{lem}\label{lemma_adjoint}\cite[Theorem 3.11]{Villegas:07}
If the boundary control and observation system \eqref{eqn:system2} is impedance energy-preserving, then its dual system $(\A_{d}, \BB_{d}, \CC_{d})$ is also impedance energy-preserving. Moreover, $(\A_{d}, \BB_{d}, \CC_{d}) = (-\A, \CC, \BB)$. 
\end{lem}

\begin{lem}\label{lem:energy_preserving}
   Assume that the boundary control and observation system \eqref{eqn:system2} is impedance energy-preserving, then it is exactly controllable and exactly observable.
\end{lem}
\begin{proof}
    We consider the operator 
    \begin{align*} 
        A_0 x &= \left( P_2 \frac{\partial^2}{\partial \zeta^2} + P_1 \frac{\partial}{\partial \zeta} + P_0 \right) \mathcal H x, \\
        D(A_0) &= \left\{ x \in X \, \vert \ \mathcal H x \in \mathrm{H^2}((0,1);\mathbb F^n),\  W_{B,1} \tau(\mathcal H x)=0 \right\}.
    \end{align*}
    As the boundary control system is impedance energy-preserving, the operator $A_0\coloneqq \A_{\mid\ker \BB}$ generates a strongly continuous unitary group \cite{GorZwaMas:2005}. Thus, $-A_0$ generates a strongly continuous bounded semigroup. Therefore, by Proposition \ref{prop:optimizability} exact controllability is equivalent to optimizability. Accordingly, we show that the system can be exponentially stabilized by a negative output feedback, that is $u(t) = -k y(t),\, k>0$. 
    Let $(T_{k}(t))$ be the semigroup generated by $\mathrm A_{k} := \mathcal A_{\mid\ker \mathcal B +k \mathcal C}$.
    Indeed, for $x_{0} \in D(\mathcal A)\cap \ker (\mathcal B +k \mathcal C)$
    \begin{equation}\label{eqn:4.3}
    \begin{aligned}
        \Re \left\langle \mathcal A x, x \right\rangle_{X} 
        &= \Re \left\langle \BB x, \CC x \right\rangle_{\mathbb R^{2n}} \\
        &= \Re \left\langle -k \CC x, \CC x \right\rangle_{\mathbb R^{2n}}  \\
        &= -k \left\|\CC x \right\|^{2}.
    \end{aligned}
    \end{equation}
    Since $W:=\begin{bsmallmatrix}
        W_{B,1} \\
        W_{C}
    \end{bsmallmatrix}$ 
    is invertible, it follows that $\|W v \|^{2} \ge m_{1} \|v\|^{2}$ for every $v \in \mathbb F^{n}$ and some $m_{1}>0$. so we have 
    \begin{align*}
        (1 +k) \|\CC x\|^{2} 
        &= \| \BB x\|^{2} + \|\CC x\|^{2} \\
        &= \left\| W \tau(\mathcal{H} x) \right\|^{2} \\
        &\ge m_{1} \|\tau(\mathcal{H} x)\|^{2} \\
        &\ge m_{1} \left( \|(\mathcal{H}x)(1)\|^{2} + \|(\mathcal{H}x)(0)\|^{2} + \|(\mathcal{H}x)'(1)\|^{2} \right) .
    \end{align*}
    Thus, by equation \eqref{eqn:4.3} we obtain that for $x_{0} \in D(\mathcal A)\cap \ker (\mathcal B +k \mathcal C)$
    \begin{align*}
        \langle \mathcal A x, x\rangle
        &= -k \|\CC x\|^{2} \\
        &\le - m_{2} \left( \|(\mathcal{H}x)(1)\|^{2} + \|(\mathcal{H}x)(0)\|^{2} + \|(\mathcal{H}x)'(1)\|^{2} \right),
    \end{align*}
    where $m_{2}=\frac{k m_{1}}{1+k}$. Thus, by Theorem 4.3.15 in \cite{Aug:16} we conclude that the semigroup $(T_{k}(t))$ generated by $\mathrm A_{k}$ is exponentially stable. This leads to a mild solution in $\mathrm{L^2}((0,\infty); X)$. Hence, the system is optimizable and thus it is exactly controllable. Finally, from Theorem 11.3.9 of \cite{Sta:05} we conclude that the system is also exactly observable.
\end{proof}

A key observation in this framework is the use of the following proposition, which states that given a well-posed boundary control and observation system, the exact controllability and exact observability remains unchanged under admissible feedback. 
\begin{prop}\label{lemma:feedback_exact_contr}\cite[Remark 6.9]{Wei:94}
    Let $F$ be an admissible feedback operator for the well-posed boundary control and observation system \eqref{eqn:system2}. Then the closed-loop boundary control and observation system generated by the operators $(\A, (\BB - F\CC), \CC)$ is exactly controllable (resp. exactly observable) if and only if the open-loop system generated by the operators $(\A, \BB, \CC)$ is exactly controllable (resp. exactly observable).
\end{prop}

We now state the main result concerning the controllability and observability properties of the boundary control system \eqref{eqn:system2}.
\begin{thm}\label{mainresult2}
     The well-posed boundary control and observation system \eqref{eqn:system2} is exactly controllable and exactly observable.
\end{thm}
\begin{proof}
    By Theorem \ref{prop0}, the boundary control and observation system 
    \begin{equation}\label{eq:phs}
    \begin{aligned}
        \frac{\partial x}{\partial t}(\zeta,t) &= \left( P_2 \frac{\partial^2}{\partial \zeta^2} + P_1 \frac{\partial}{\partial \zeta} + P_0(\zeta) \right)\mathcal{H}(\zeta) x(\zeta,t), \quad \zeta \in [0,1] \\
        u_{e}(t) &= \frac{1}{\sqrt{2}}
        \begin{bmatrix}
            P_2 \frac{\partial}{\partial\zeta} (\mathcal{H}x)(1,t) + \tfrac{1}{2} P_1 (\mathcal{H}x)(1,t) \\  
            P_2 \frac{\partial}{\partial\zeta} (\mathcal{H}x)(0,t) + \tfrac{1}{2} P_1 (\mathcal{H}x)(0,t)
        \end{bmatrix}, \\
    y_{e}(t) &= \frac{1}{\sqrt{2}} 
    \begin{bmatrix}
        (\mathcal{H}x)(1,t)  \\  
        -(\mathcal{H}x)(0,t)
    \end{bmatrix},
    \end{aligned}
    \end{equation}
    is impedance energy-preserving. Then, it follows from Lemma \ref{lem:energy_preserving} that it is exactly controllable and exactly observable. Since the system \eqref{eqn:system2} is well-posed, then $B^{-1}$ is invertible. Using 
    \begin{align*}
         u(t) &=  B_1 u_{e}(t) + B_2 y_{e}(t), \\
          y(t) &= C_1  u_{e}(t) + C_2  y_{e}(t),
    \end{align*} 
    we see that the system \eqref{eqn:system2} is the closed-loop system of \eqref{eq:phs} via the feedback defined by $-B_1^{-1} B_2$. Since the feedthrough term of system \eqref{eq:phs} is zero, then $B_1^{-1} B_2$ is an admissible feedback operator. Thus, from Lemma \ref{lemma:feedback_exact_contr}, we deduce that the closed-loop system, that is the system
    \begin{align*}
        \frac{\partial x}{\partial t}(\zeta,t) &= \left( P_2 \frac{\partial^2}{\partial \zeta^2} + P_1 \frac{\partial}{\partial \zeta} + P_0(\zeta) \right)\mathcal{H}(\zeta) x(\zeta,t), \quad t>0, \quad \zeta \in [0,1] \\
        u(t) &= W_{B,1} \tau(\mathcal H x), \quad t>0 \\
        y(t) &= W_C \tau(\mathcal H x), \quad t>0.
    \end{align*}
    is exactly controllable and exactly observable.
\end{proof}

\section{Examples}\label{section5}
In this section we apply our theory to two Euler-Bernoulli beam models.
\subsection{Euler-Bernoulli beam with viscous air damping} 
We study an Euler-Bernoulli beam with viscous air damping. In this model, damping is assumed proportional to the transverse velocity, $\frac{\partial \omega}{\partial t}$ \cite{BanksInman_91}. In this case the vibrations of the beam are described by 
\begin{equation}\label{eq:EB_vis}
    \rho(\zeta) \frac{\partial^2 \omega}{\partial t^2}(\zeta,t) + \frac{\partial^2}{\partial \zeta^2}\left( EI(\zeta) \frac{\partial^2 \omega}{\partial \zeta^2}(\zeta,t)\right) + \gamma(\zeta) \frac{\partial \omega}{\partial t}(\zeta,t) =0, \quad t>0, \quad \zeta \in (0,1).
\end{equation}
Here, $\omega(\zeta,t)$ is the vertical displacement, $\rho(\zeta)$ denotes the mass density per unit length, $E I(\zeta) > 0$ is the flexural rigidity (or bending stiffness) of the beam while $E(\zeta)>0$ is the elasticity modulus and $I(\zeta)>0$ is the moment of inertia of the cross section. The non-negative coefficient $\gamma(\zeta)$ represents the viscous external damping.
Associated with \eqref{eq:EB_vis} we choose the following boundary conditions
\begin{align*}
    \frac{\partial^2 \omega}{\partial t \partial \zeta} (0,t) &= 0, \quad \frac{\partial^2 \omega}{\partial t \partial \zeta}(1,t) = 0, \\
     \left[ \frac{\partial}{\partial \zeta} \left( EI(\zeta) \frac{\partial^2 \omega}{\partial \zeta^2}(\zeta,t) \right) \right]_{\zeta=0} &=0, \quad  
     \left[ \frac{\partial}{\partial \zeta}\left(EI(\zeta) \frac{\partial^2 \omega}{\partial \zeta^2}(\zeta,t) \right) \right]_{\zeta=1} = u(t),
\end{align*}
and measurement of the output 
$$ y(t) = \frac{\partial \omega}{\partial t}(1,t) . $$
By defining $x=\begin{bsmallmatrix}
    \rho \frac{\partial \omega}{\partial t} \\ \frac{\partial^2 \omega}{\partial \zeta^2}
\end{bsmallmatrix}$, and choosing
$$ P_2= \begin{bmatrix}
    0 & -1 \\ 
    1 & 0
\end{bmatrix}, \quad P_1 =0, \quad P_0= \begin{bmatrix}
    -\gamma & 0 \\
    0 & 0 
\end{bmatrix}, \text{ and} \quad \mathcal{H}=\begin{bmatrix}
    \frac{1}{\rho} & 0 \\
    0 & E I
\end{bmatrix}, $$
the damped Euler-Bernoulli beam equations \eqref{eq:EB_vis} belong to the class of systems \eqref{eqn:system1} with the following boundary control and observation
\begin{align*}
    \begin{bmatrix}
        0 \\ I 
    \end{bmatrix} u(t) &= \begin{bmatrix}
        0 & 0 & 0 & 0 & 0 & 0 & 1 & 0 \\
        0 & 0 & 1 & 0 & 0 & 0 & 0 & 0 \\
        0 & 0 & 0 & 0 & 0 & 0 & 0 & 1 \\
        0 & 0 & 0 & 1 & 0 & 0 & 0 & 0 \\
    \end{bmatrix} \begin{bmatrix}
        (\mathcal H x)(1,t) \\ (\mathcal H x)'(1,t) \\ (\mathcal H x)(0,t) \\ (\mathcal H x)'(0,t)  
    \end{bmatrix}, \\
    y(t) &= \begin{bmatrix}
         1 & 0 & 0 & 0 & 0 & 0 & 0 & 0
    \end{bmatrix} \begin{bmatrix}
        (\mathcal H x)(1,t) \\ (\mathcal H x)'(1,t) \\ (\mathcal H x)(0,t) \\ (\mathcal H x)'(0,t) 
    \end{bmatrix}. 
\end{align*}
We have 
\begin{equation*}
    \begin{bmatrix}
    K_1 \\ B_1 
\end{bmatrix} = \begin{bmatrix}
    0 & 0 & 0 & 1 \\
    0 & 1 & 0 & 0 \\
    0 & 0 & -1 & 0 \\
    -1 & 0 & 0 & 0
\end{bmatrix}. 
\end{equation*} 
We clearly have that $\begin{bsmallmatrix}
    K_1 \\ B_1
\end{bsmallmatrix}$ is invertible. Thus, by Corollary \ref{corollary1}, with this choice of input and output, the system is well-posed. 

\subsection{Beam with distributed elastic support} 
We consider an Euler-Bernoulli beam where its end are supported transversely by distributed elastic springs where $k_{t}$ denotes the translational stiffness and $k_{r}$ denotes the rotational stiffness of the spring, see \cite{mathews1958vibrations, ElasticBeam:19}. The equation of motion for this system is given by  
\begin{equation}\label{EB:elastic}
    \rho(\zeta) \frac{\partial^2 \omega}{\partial t^2}(\zeta,t) + \frac{\partial^2}{\partial \zeta^2}\left( EI(\zeta) \frac{\partial^2 \omega}{\partial \zeta^2}(\zeta,t)\right) + k \omega (\zeta,t) =0, \quad t>0, \quad \zeta \in (0,1).
\end{equation}
The Euler-Bernoulli beam equation \eqref{EB:elastic} may be written in the formulation of the class of systems \eqref{eqn:system1}
\begin{equation*}
\resizebox{\textwidth}{!}{$
\frac{\partial}{\partial t} 
    \begin{bmatrix}
    \rho \frac{\partial \omega}{\partial t} \\ 
    \frac{\partial^2 \omega}{\partial \zeta^2} \\
    \omega
\end{bmatrix} =
\begin{bmatrix}
    0 & -1 & 0 \\
    1 & 0 & 0 \\
    0 & 0 & i 
\end{bmatrix} \frac{\partial^2}{\partial \zeta^2} 
\begin{bmatrix}
    \frac{1}{\rho} & 0 & 0 \\
    0 & EI & 0 \\
    0 & 0 & k
\end{bmatrix}
\begin{bmatrix}
    \rho \frac{\partial \omega}{\partial t} \\ 
    \frac{\partial^2 \omega}{\partial \zeta^2} \\ 
    \omega
\end{bmatrix} +
\begin{bmatrix}
    0 & 0 & -1 \\
    0 & 0 & 0 \\ 
    1 & \frac{-i k}{EI} & 0
\end{bmatrix} \begin{bmatrix}
     \frac{1}{\rho} & 0 & 0 \\
    0 & EI & 0 \\
    0 & 0 & k
\end{bmatrix} \begin{bmatrix}
    \rho \frac{\partial \omega}{\partial t} \\ 
    \frac{\partial^2 \omega}{\partial \zeta^2} \\ 
    \omega
\end{bmatrix}. 
$}
\end{equation*} 
We choose the following boundary conditions
\begin{align*}
    EI(0) \frac{\partial^2 \omega}{\partial \zeta^2}(0,t) + k_{r} \frac{\partial \omega}{\partial \zeta}(0,t) &= 0, \qquad
    \frac{\partial^2 \omega}{\partial t \partial \zeta} (0,t) = 0, \\
    \left[ \frac{\partial}{\partial \zeta}\left(EI(\zeta) \frac{\partial^2 \omega}{\partial \zeta^2}(\zeta,t) \right) \right]_{\zeta=0} + k_{t} w(0,t) &= 0, \qquad
    \frac{\partial^2 \omega}{\partial t \partial \zeta}(1,t) = 0, \\
     EI(1) \frac{\partial^2 \omega}{\partial \zeta^2}(1,t) + k_{r} \frac{\partial \omega}{\partial \zeta}(1,t) &= 0, \\
     \left[ \frac{\partial}{\partial \zeta}\left(EI(\zeta) \frac{\partial^2 \omega}{\partial \zeta^2}(\zeta,t) \right) \right]_{\zeta=1} + k_{t} \omega(1,t) &= u(t) .
\end{align*}
We introduce the output 
\begin{equation*}
 y(t) = \frac{\partial \omega}{\partial t}(1,t) .
\end{equation*}
We have that 
\begin{equation*}
\begin{bmatrix}
    K_1 \\ B_1 
\end{bmatrix} = \begin{bmatrix}
    0 & 0 & 0 & 0 & 0 & -i k_{r} \\
    0 & 0 & 0 & 0 & 1 & 0 \\
    0 & 0 & 0 & 1 & 0 & 0 \\
    0 & 1 & 0 & 0 & 0 & 0 \\
    0 & 0 & -i k_{r} & 0 & 0 & 0 \\
    1 & 0 & 0 & 0 & 0 & 0
\end{bmatrix}.
\end{equation*}
Since $\begin{bsmallmatrix}
    K_{1} \\
    B_{1}
\end{bsmallmatrix}$ is invertible, by Corollary \ref{corollary1}, the system is well-posed. 

\section{Conclusion and perspectives}\label{section6}
We have investigated the well-posedness of a class of boundary control and observation system in the form \eqref{eqn:system1} which in particular includes Euler-Bernoulli beam models. To this end, we have provided a necessary and sufficient condition for well-posedness. We have showed that, under full control and observation, well-posedness implies exact controllability. A perspective could be to extend the proposed approach to a more general class that includes the Rayleigh beam equation.

\bibliographystyle{siamplain}
\bibliography{references}

\end{document}